\documentclass{amsart}
\usepackage{graphicx}
\usepackage{amssymb}
\usepackage{amsfonts}
\usepackage{hyperref}
\usepackage{mathrsfs}
\swapnumbers
\newtheorem{theorem}{Theorem}[subsection]
\newtheorem{lemma}[theorem]{Lemma}
\newtheorem{corollary}[theorem]{Corollary}

\newtheorem{proposition}[theorem]{Proposition}
\theoremstyle{definition}

\newtheorem{assumption}[theorem]{Assumption}
\newtheorem{remark}[theorem]{Remark}

\newtheorem*{acknowledgment}{Acknowledgment}

\numberwithin{equation}{section}
\theoremstyle{plain}

\numberwithin{equation}{section} 
\numberwithin{figure}{section} 
\theoremstyle{plain}
\theoremstyle{plain}
\theoremstyle{remark}
\newtheorem*{acknowledgement*}{Acknowledgement}
\theoremstyle{example}

\newcommand{\cA}{{\mathcal A}}

\newcommand{\cC}{{\mathcal C}}

\newcommand{\cF}{{\mathcal F}}
\newcommand{\cG}{{\mathcal G}}
\newcommand{\cH}{{\mathcal H}}

\newcommand{\cL}{{\mathcal L}}
\newcommand{\cM}{{\mathcal M}}

\newcommand{\cP}{{\mathcal P}}
\newcommand{\cQ}{{\mathcal Q}}
\newcommand{\cR}{{\mathcal R}}

\newcommand{\te}{{\theta}}

\newcommand{\Om}{{\Omega}}
\newcommand{\om}{{\omega}}
\newcommand{\ve}{{\varepsilon}}
\newcommand{\del}{{\delta}}
\newcommand{\Del}{{\Delta}}
\newcommand{\gam}{{\gamma}}
\newcommand{\Gam}{{\Gamma}}

\newcommand{\Sig}{{\Sigma}}
\newcommand{\sig}{{\sigma}}
\newcommand{\al}{{\alpha}}
\newcommand{\be}{{\beta}}
\newcommand{\ka}{{\kappa}}
\newcommand{\la}{{\lambda}}
\newcommand{\Lam}{{\Lambda}}

\newcommand{\bbC}{{\mathbb C}}

\newcommand{\bbN}{{\mathbb N}}
\newcommand{\bbP}{{\mathbb P}}
\newcommand{\bbR}{{\mathbb R}}

\newcommand{\bbZ}{{\mathbb Z}}
\newcommand{\bbI}{{\mathbb I}}

\begin{document}
\title[]{Limit theorems for random non-uniformly expanding or hyperbolic maps with exponential tails}  
 \vskip 0.1cm
 \author{Yeor Hafouta \\
\vskip 0.1cm
Department  of Mathematics\\
The Ohio State University}
\email{yeor.hafouta@mail.huji.ac.il, hafuta.1@osu.edu}%

\thanks{ }
\dedicatory{  }
 \date{\today}

\maketitle
\markboth{}{Limit theorems for random towers} 
\renewcommand{\theequation}{\arabic{section}.\arabic{equation}}
\pagenumbering{arabic}

\begin{abstract}\noindent
We prove a Berry-Esseen theorem, a local central limit theorem and (local) large and (global) moderate deviations principles for i.i.d. (uniformly) random non-uniformly expanding or hyperbolic maps with exponential first return times. 
 Using existing results  the problem is reduced to certain random  (Young) tower extensions,  which is the main focus of this paper. On the random towers we will obtain our results using contraction properties of random complex equivariant cones with respect to the complex Hilbert projective metric.
\end{abstract}

\section{Introduction}\label{sec1}
Limit theorems for deterministic  expanding or hyperbolic dynamical systems is a well studied topic. Such results are often proven using  spectral properties of an underlying family of complex transfer operators, what these days is often referred to as the Nagaev-Guivar\'ch method (see \cite{GH,HH}). Since then there were several extensions to certain classes of non-uniformly expanding or hyperbolic deterministic dynamical systems (see \cite{GO,MN} and  references therein), where the most general approach is based on tower extensions in the sense of Young \cite{Y1,Y2}.

A random dynamical system is generated  by a  probability (or measure) preserving system system $(\Om,\cF,\bbP,\sig)$, and
a family of maps $f_\om, \om\in\Om$.  The  random orbit of a point $x$ is generated by compositions $f_\om^n x=f_{\sig^{n-1}\om}\circ\cdots f_{\sig\om}\circ f_\om x$ of these maps along trajectories of the ``driving" system $(\Om,\cF,\bbP,\sig)$.
One of the first authors to study limit theorems for random dynamical systems is Kifer \cite{K96,K98} which, in particular, proved large deviations principles and central limit theorems for several classes of random uniformly expanding maps.
Recently (see \cite{Aimino, BB, DZ, DFGTV1, DFGTV2, DFGTV3, DH, HK, HeSt} and references therein) there has  been a growing interest in additional limit theorems for random expanding or hyperbolic dynamical systems. We also refer to \cite{Arno, Conze, HafSDS, Hyd, NSV} for central limit theorems for some classes of time dependent (sequential) dynamical systems which are not necessarily random.
In particular, in \cite{DFGTV2, HK} a local central limit theorem (LCLT) was proven for the first time in the context of random (expanding) dynamical systems, while in \cite{HK} a Berry-Esseen theorem was also proven for the first time in the random expanding case. In \cite{DFGTV3} the authors proved an LCLT for some classes of random Anosov maps, while in \cite{DH1}, together with the first author of \cite{DFGTV2}  we extended the Berry-Esseen theorem for such maps.
Both approaches were based on certain (different) types of spectral method for complex random operators.


Limit theorems for random   non-uniformly expanding or hyperbolic maps are still not fully studied. In \cite{BBM} the authors presented the notion of a random Young tower, showed that certain classes of random i.i.d. unimodel maps admit a random tower extension and obtained almost sure rates of mixing (decay of correlations). Results in this direction were also obtained later by several authors \cite{ABR, BB, BBD, BBR, Du}.  In \cite{Su} the author proved an almost sure invarinace principle (ASIP) for random Young towers. While the ASIP is a powerful statistical tool which is much stronger than the usual CLT, it does not imply the fine limit theorems studied in this paper. 

In this manuscript we will prove a Berry-Esseen theorem, a local central limit theorems and large and moderate deviations principles for maps which admit a random (uniform) tower extension, with exponential tails. Our results will be applicable then to i.i.d. uniformly random non-uniformly expanding or hyperbolic maps with exponential first return times. In the partially expanding case the limit theorems hold true when the initial measure is $\mu_\om$ is equivalent to the Lebesgue measure and $(f_\om)_*\mu_\om=\mu_{\sig\om}$ (i.e. $\mu_\om$ is an equivariant family\footnote{in the terminology of \cite{BBD}\, $\mu_\om$ are ``sample stationary measures".}), while in the partially hyperbolic case $\mu_\om$ is an equivariant family of physical measures. 
For the best of our knowledge the are no other results in this direction even for specific cases with exponential tails. Our approach here is spectral; generalizing the ideas in \cite{M-D}, we construct random real Birkhoff cones and show that the (appropriately floor-wise normalized) random transfer operators on the random tower are projective contractions of these cones (with respect to the corresponding Hilbert metrics). Then we apply the complex conic-perturbations theory of Rugh \cite{Rug} (see also \cite{Dub1,Dub2}) and show that appropriate complex perturbation of the above random transfer operators strongly contract the canonical complexification of these cones. Applying a general result from \cite{HK} which extends Rugh's complex spectral gap theory to compositions of random complex operators, will result in a random complex Ruelle-Perron-Frobenius (RPF) theorem. Once this theorem is established the limit theorems are derived using ideas from \cite[Ch. 7]{HK} (the relevant arguments share some similarities with the arguments in \cite{CP} for deterministic subshifts of finite type).

The paper is organized as follows. In Section \ref{Sec1} we will present the main results (limit theorems) for random Young towers, while in Section \ref{SecApp} we will present  our main applications to random partially expanding or hyperbolic maps. In Section \ref{Sec4} we will prove a few results concerning random transfer operators, partitions and cones on random towers. We will prove there a random Lasota-Yorke type inequality for random complex transfer operators generated by the Jacobian of the tower map, and construct certain types of random partitions. Using these partitions,  we define random real Birkhoff cones, show that the  complex transfer operators mentioned above are strong contractions of the canonical complexification of these cones, and derive the RPF theorem. Section \ref{SecLimThms} is devoted to application of this RPF theorem to limit theorems. 

\section{Preliminaries and main results}\label{Sec1}

\subsection{Random Young towers}\label{Pre}
Let $\cP_0=(\Om_0,\cF_0,P_0)$ be a probability space and let $\cP=\cP^\bbZ=(\Om,\cF,\bbP)$ be the appropriate product space. Let $\sig:\Om\to\Om$ be the left shift given by $\sig\om=(\om_{n+1})_{n\in\bbZ}$, where $\om=(\om_n)_{\in\bbZ}$.
 Let $(M,\cM)$ be a measurable space. Our setup consist of a family of measurable sub-spaces $M_\om\subset M$ and maps $f_\om:M_\om\to M_{\sig\om}$, where $f_\om=f_{\om_0}$ depends only on the $0$-th coordinate of $\om=(\om_k)_{k\in\bbZ}$ (so the random maps $f_{\sig^n\om},\,n\geq0$ are independent).
 Moreover, there are measurable subsets $\Del_{\om,0}$ of $M_\om$ and countable measurable partition $\{\Lambda_{\om,i}\}$ of $\Del_{\om,0}$ so that for any $\om$ and $i$ there is a minimal positive integer $R_{\om,i}$ such that
\[
f_\om^{R_{\om,i}}\Del_{\om,i}\subset\Del_{\sig^{R_{\om,i}}\om,0}
\] 
where for each $n$ we define 
$f_\om^n=f_{\sig^{n-1}\om}\circ\cdots\circ f_{\sig\om}\circ f_\om$. Furthermore, $f_\om^{R_{\om,i}}|\Lambda_{\om,i}\to\Del_{\sig^{R_{\om,i}},0}$ is a measurable bijection for each $i$.
Our measurability assumption are as follows. We assume that the map $\om\to R_{\om,i}$ is measurable for each $i$, that  the sets $M_\om$ and $\Lambda_{\om,i},\,i\in\bbN$ are measureable in $\om$ in the sense of \cite[Section 3]{Cr}, and that the map $(\om,x)\to f_\om(x)$ is measureble in both $\om$ and $x$ with respect to the $\sig$-algebra on the skew-product space $\{(\om,x):\,\om\in\Om, \,x\in M_\om\}$ induced from the product $\sig$-algebra $\cF\times \cM$.
 
Next,  for any fixed $\om$ we view $\{R_{\om,i}\}$ as a function $R_\om:\Del_{\om,0}\to\bbN$ by defining $R_\om|_{\Lam_{\om,i}}\equiv R_{\om,i}$.
We define now a random tower $\Del_\om=\cup_{\ell\geq0}\Del_{\om,\ell}$ as follows: for any $\ell\geq1$ we set 
\[
\Del_{\om,\ell}=\{(x,\ell):\,x\in\Del_{\sig^{-\ell}\om,0}, R_{\sig^{-\ell}\om}(x)\geq\ell+1\}.
\]
We will also identify between $\Del_{\om,0}$ and $\Del_{\om,0}\times\{0\}$.
The above partitions induce a partition $\cQ_\om=\{\Del_{\om,\ell,i}: (\ell,i)\in\cG_\om\}$ of $\Del_{\om}$, where 
$\Del_{\om,\ell,i}=\Lambda_{\sig^{-\ell}\om,i}\times\{\ell\}$ and $\cG_\om$ is the set of pairs $(\ell,i)$ so that  $R_{\sig^{-\ell}\om,i}>\ell$.

We define a map $F_\om:\Del_\om\to \Del_{\sig\om}$ by 
\[
F_\om(x,\ell)=\begin{cases}(x,\ell+1) & \text{ if }R_{\sig^{-\ell}\om}(x)>\ell+1 \\(f^{\ell+1}_{\sig^{-\ell}\om}x,0)& \text{ if }R_{\sig^{-\ell}\om}(x)=\ell+1\end{cases}.
\]
For any $n\geq1$, the $n$-th order ``cylinder" partition of $\Del_\om$ is given by 
\[
C_{\om,n}=\bigvee_{i=0}^{n-1}\left(F_\om^{i}\right)^{-1}\cQ_{\sig^i\om}
\]
where 
\[
F_\om^i=F_{\sig^{i-1}\om}\circ\cdots\circ F_{\sig\om}\circ F_\om.
\]
Given a point $x\in\Del_\om$ we denote by $C_{\om,n}(x)$ the unique $n$-th order cylinder containing $x$. Then the cylinder $C_{\om,n}(x)$ depends only on $C_{\om,1}(x)$ and the sets $\Lambda_{\sig^j\om,i_j},\,1\leq j<n$ so that $F^j_{\om} x\in\Lambda_{\sig^j\om,i_j}\times\{0\}$.
We define a separation time on $\Del_\om$ by setting\footnote{In terms of the maps $\{f_\om\}$, on the $\ell$-th level of the tower $\Del_\om$ we have that $s_\om(x,y)+\ell$ is the time the random orbit of $x_0$ and $y_0$ stays together in the sense that all the returns to the random bases occur thorough the same atom, where $x=(x_0,\ell)$ and $y=(y_0,\ell)$.}  $s_\om(x,y)$, $x,y\in\Del_\om$ to be the first time $n$ so that $x$ and $y$ do not belong to the same partition element in $\cC_{\om,n}$ (when there is no such $n$ we set $s_\om(x,y)=\infty$). 
We assume that the partition $C_{\om}=\bigvee_n C_{\om,n}$ separates point in the sense that
$$
\bigvee_n C_{\om,n}(x)=\{x\}.
$$

Next, let $m_\om$ be a family of probability measures on $\Del_{\om,0}$ so that with some $C>0$ for $\bbP$-almost all $\om$ we have 
\begin{equation}\label{A1}
\sum_{\ell=0}^\infty m_{\sig^{-\ell}_\om}(R_{\sig^{-\ell}\om}\geq \ell)\leq C.
\end{equation}
This family induces a finite uniformly bounded measure $\textbf{m}_\om$ on $\Del_\om$ by identifying $\Lambda_{\om,\ell,i}$ with $\Lambda_{\sig^{-\ell}\om,i}$. 
Henceforth, when there is no ambiguity, we will write $m_\om$ instead of $\textbf{m}_\om$. 

Let $JF_\om$ be the Jacobian corresponding to the map $F_\om:(\Del_\om,m_\om)\to (\Del_{\sig\om},m_{\sig_\om})$. Then $JF_\om$ equals $1$ on points $(x,\ell)$ so that $F(x,\ell)=(x,\ell+1)$.
 Let $\be\in(0,1)$.  We assume that there is a constant $A_1>0$ so that any $\ell\geq0$ and $x=(x_0,\ell),y=(y_0,\ell)\in\Del_{\om,\ell,i}$ with $R_{\sig^{-\ell}\om,i}=\ell+1$ we have
\begin{equation}\label{Distortion0}
\left|\frac{JF_\om x}{JF_\om y}-1\right|=\left|\frac{Jf_{\om_{-\ell}}^{R_{\om_{-\ell}}} x_0}{Jf_{\om_{-\ell}}^{R_{\om_{-\ell}}} y_0}-1\right|\leq A_1\be^{s_{\sig\om}(F_\om x,F_\om y)}
\end{equation}
where $\om_{-\ell}=\sig^{-\ell}\om$.
\begin{theorem}[Theorem 2.5 (i) in \cite{ABR}]\label{Thm 2.5 1}
There exists a strictly positive function $h_\om:\Del_\om\to\bbR$ 
and constants $c_0,c_1,c_2>0$ so that $\mathbb P$-almost surely $c_0\leq\inf h_\om\leq\sup h_\om\leq c_1$ and 
$|h_\om(x)-h_\om(y)|\leq c_2\be^{s_\om(x,y)}$ for all $x,y\in\Del_\om$. Moreover, $\int h_\om dm_\om=1$  and the family of measures $\mu_\om=h_\om dm_\om$ satisfies  $(F_\om)_*\mu_{\om}=\mu_{\sigma\om}$.
\end{theorem}
Under the assumptions presented in the next section the family of measures $\mu_\om$ is the unique family of absolutely continuous probability measures satisfying  $(F_\om)_*\mu_{\om}=\mu_{\sigma\om}$.

\subsection{Limit theorems: main results}
\subsubsection{Main assumptions}
Let $\varphi_\om:\Del_\om\to\bbR$, $\om\in\Om$ be a family of functions such that $\varphi(\om,x)=\varphi_{\om}(x)$ is measurable in both $\om$ and $x$
and for some $C_1,C_2>0$ for $\bbP$-almost every $\om$ and all $x,y\in\Del_\om$ we have
$$
|\varphi_\om(x)|\leq C_1\,\,\text{ and }\,\,|\varphi_\om(x)-\varphi_\om(y)|\leq C_2\beta^{s_\om(x,y)}.
$$
 For $\bbP$-almost all $\om$ we consider the functions
\[
S_n^\om \varphi=\sum_{j=0}^{n-1}\varphi_{\sig^j\om}\circ F_\om^j.
\]
In this section we will view $S_n^\om\varphi(x)$ as a sequence of random variables when $x$ is distributed according to either $\mu_\om$ or $m_\om/m_\om(\Del_\om)$. 

We will obtain our results under the following.
\begin{assumption}\label{Ass Ap}[Aperiodicity of return times]
There are $N_0$ and $t_1,t_2,...,t_{N_0}\in\bbN$ such that $\gcd\{t_i\}=1$ and $\bbP$-a.e. 
$m_\om(R_{\om}=t_i)>0$;  Moreover, $R_\om$ is a stopping time, namely the map $(\om,x)\to R_{\om}(x)$ is measurable and if $R_{\om}(x)=n$ then also $R_{\om'}(x)=n$, where $\om'$ is a sequence whose first $n$ coordinates are the same as $\om$.
\end{assumption}
\begin{assumption}\label{Ass Exp}[Exponential tails]
There are $c_1,c_2>0$ so that for ll $n\geq1$ and a.e. $\om$,
\begin{equation}\label{Exp}
m_\om(R_\om\geq n)\leq c_1e^{-c_2n}.
\end{equation}
\end{assumption}

We will also need the following assumption. 
\begin{assumption}[Uniform ``lower randomness"]\label{AsCov}
For any $\ve>0$ there are $J\in\bbN$ and $\del>0$ so that for $\bbP$-a.a. any $\om$ there are atoms $Q_{\om,i}=\Del_{\om,\ell_i(\om),j_i(\om)}$, $1\leq i\leq k_\om\leq J$ so that for all $i$,
\[
m_\om(Q_i)\geq\del
\]
and with $Q=\Del_\om\setminus(Q_1\cup Q_2\cup\cdots\cup Q_{k_\om})$, 
\begin{equation}\label{ve '}
\del\leq m_\om(Q)<\ve.
\end{equation}
\end{assumption}
\begin{remark}\label{RemApp}
In our applications in Section \ref{SecApp} we will use one of the following.

(i) Assumption \ref{AsCov} holds true in the following situation. Let us order the atoms of partition into cylinders of length $1$ according to their $\tilde m_\om$-measure. Let us denote by $Q_{\om,1},Q_{\om,2},...$ the ordered atoms. Then the condition holds true if the series $\sum_{j=1}^{\infty}\tilde m_\om(Q_{\om,i})$ converge uniformly in $\om$ and for any $i$,
\begin{equation}\label{Cond1}
\text{ess-inf } m_\om(Q_{\om,i})>0. 
\end{equation}
Let $\cR_{i,\om}$ be the return time corresponding to $Q_{\om,i}$. Then the ratio between $m_\om(Q_{\om,i})$ and $1/J f_{\sig^{-\ell}\om}^{\cR_{i,\om}}(x_0)$ is bounded and bounded away from $0$,  where $x=(x_0,\ell)$ is an arbitrary point in $A_{\om,i}$. Thus the assumption holds true if the Jacobian appearing in the above denominator is bounded from above uniformly in $i$.
\vskip.2cm

(ii) Assumption \ref{AsCov} holds also holds true when the tails $m_\om(R_\om\geq\ell)$ decay uniformly in $\om$ to $0$ as $\ell\to\infty$, the Jacobian (or the derivative) of $f_\om$ is uniformly bounded in $\om$ on $\Lam_{\om,i}$ for each $i$ (so that the measure of an atom $\Del_{\om,i}$ such that $R_{\om,i}\leq\ell$ is larger than some $\del_\ell>0$ which depends only on $\ell$) and for every $\ell$ large enough there is $k_\ell$ so that for $\bbP$ a.e. $\om$ the set $\{\ell<R_\om\leq \ell+k_\ell\}$ is nonempty.  
\end{remark}

As usual, in order to start describing the distributional limiting behavior of the random Birkhoff sums we need the following.
\begin{theorem}
Under Assumptions  \ref{Ass Ap}, \ref{Ass Exp} and \ref{AsCov}, 
there is number $\Sig^2\geq0$ so that $\bbP$-a.e. we have
\[
\Sig^2=\lim_{n\to\infty}\frac1n\text{Var}_{\mu_\om}(S_n^\om \varphi).
\]
Moreover, let $\mu$ be the measure with fibers $\mu_\om$, namely $\mu=\int \mu_\om dP(\om)$. Then 
$\Sig^2=0$ if and only if there is a function $r(\om,x)\in L^2(\mu)$ so that $\mu$-a.s. we have
\[
\varphi_\om(x)-\mu_\om(\varphi_\om)=r(\sig\om, F_\om x)-r(\om,x)=r(T(\om,x))-r(\om,x)
\]
where $T(\om,x)=(\sig\om,F_\om x)$ is the corresponding skew product map. Furthermore, when $\Sig^2>0$ then the sequence $\left(S_n^\om \varphi-\mu_\om(S_n^\om \varphi)\right)/\sqrt n$ converges in distribution towards a centered normal random variables with variance $\Sig^2$.
\end{theorem}
This theorem follows from \cite[Theorem 2.3]{K98} together with Theorem \ref{RPF} in the present manuscript.  We note that the theorem also holds true when the tails are of order $o(n^{-3-\del})$ for some $\del>0$, but since we need the exponential tails to prove our main results we prefer to focus on the exponential case.
\begin{remark}
By \cite{EagHaf} and (\ref{ExpCor}) we get the CLT also when the initial measure is $\bar m_\om:=m_\om/m_\om(\Del_\om)$ (in this case the mean and the variance are taken with respect to $\bar m_\om$, as well). 
\end{remark}

\subsubsection{A Berry-Esseen theorem and a local CLT}
Our first result here is optimal convergence rate in the self-normalized version of the above CLT.
 
\begin{theorem}[A Berry-Esseen theorem]\label{BE THM}
Under Assumptions  \ref{Ass Ap}, \ref{Ass Exp} and \ref{AsCov} we have the following.
\vskip0.2cm
(1) Set $\Sig_{\om,n}=\sqrt{\text{Var}_{\mu_\om}(S_n^\om\varphi)}$.
There is a random variable $c_\om>0$ so that $\bbP$-a.s. for every $n\geq1$ we have
\[
\sup_{t\in\bbR}\left|\mu_\om\left\{x: S_n^\om \varphi(x)-\mu_\om(S_n^\om \varphi)\leq t\Sig_{\om,n}\right\}-\frac{1}{\sqrt{2\pi}}\int_{-\infty}^{t}e^{-t^2/2}dt\right|\leq c_\om n^{-1/2}.
\]
\vskip0.2cm
(2) Let $v_{\om, n}$ denote the variance of $S_n^\om \varphi$ with respect to the reference measure $\bar m_\om=m_\om/m_\om(\Del_\om)$. Then 
\begin{equation}\label{VarDif}
\text{ess-sup}\sup_n|v_{\om, n}-\Sig_{\om,n}^2|<\infty
\end{equation}
and there is a random variable $d_\om>0$ so that $\bbP$-a.s. for all $n\geq1$ we have
\[
\sup_{t\in\bbR}\left| \bar m_\om\left\{x: S_n^\om \varphi(x)- \bar m_\om(S_n^\om \varphi)\leq t\sqrt{v_{\om,n}}\right\}-\frac{1}{\sqrt{2\pi}}\int_{-\infty}^{t}e^{-t^2/2}dt\right|\leq d_\om n^{-1/2}.
\]
\end{theorem}

Our next result is a local central limit theorem (LCLT). Let us begin with a formulation which is suitable for aperiodic cases.

\begin{theorem}[LCLT, aperiodic case]\label{LLT1}
Let Assumptions  \ref{Ass Ap}, \ref{Ass Exp} and \ref{AsCov} hold.
Suppose also that $\bbP$-a.s. for every compact set $J\subset\bbR\setminus\{0\}$ we have
\begin{equation}\label{A3}
\lim_{n\to\infty}\sqrt{n}\sup_{t\in J}|\mu_\om(e^{it S_n^\om \varphi})|=0.
\end{equation}
Then $\bbP$-a.s. for any continuous function $g:\bbR\to\bbR$ with compact support (or an indicator of a finite interval) we have
\[
\lim_{n\to\infty}\sup_{t\in\bbR}\left|\sqrt{2\pi n}\Sig\int g(S_n^\om \varphi(x)-\mu_\om(S_n^\om \varphi)-t)d\mu_\om(x)-e^{-\frac{t^2}{2n\Sig^2}}\int_{-\infty}^\infty g(x)dx\right|=0.
\]
The same result holds true with $\bar m_\om$ in place of $\mu_\om$ assuming that (\ref{A3}) holds true with $ \bar m_\om$.
\end{theorem}
Note that condition (\ref{A3}) excludes the case that $S_n^\om \varphi$ take valued in some lattice $\bbZ h=\{kh:\,k\in\bbZ\}$, $h>0$. In the lattice case we have the following.

\begin{theorem}[LCLT, lattice case]\label{LLT2}
Let Assumptions  \ref{Ass Ap}, \ref{Ass Exp} and \ref{AsCov} hold.
Suppose also that there is an $h>0$ so that $S_n^\om \varphi\in h\bbZ$ for any $n$ and $\bbP$-almost all $\om$. Assume also that $\bbP$-a.s. for every compact set $J\subset[-\pi/h,\pi/h]\setminus\{0\}$ we have
\begin{equation}\label{A3'}
\lim_{n\to\infty}\sqrt{n}\sup_{t\in J}|\mu_\om(e^{it S_n^\om \varphi})|=0.
\end{equation}
Then $\bbP$-a.s. for any continuous function $g:\bbR\to\bbR$ with compact support (or an indicator of a finite interval) we have
\[
\lim_{n\to\infty}\sup_{k\in\bbZ}\left|\sqrt{2\pi n}\Sig\int g(S_n^\om \varphi(x)-\mu_\om(S_n^\om \varphi)-kh)d\mu_\om(x)-e^{-\frac{(kh)^2}{2n\Sig^2}}\sum_{m\in\bbZ}g(mh)\right|=0.
\]
The same result holds true with $\bar m_\om$ in place of $\mu_\om$ assuming that (\ref{A3'}) holds true with $\bar m_\om$.
\end{theorem}
We refer the readers' to Section \ref{SecVer} for a discussion about the verification of  conditions (\ref{A3}) and (\ref{A3'}).

\subsubsection{Large and moderate deviations principles}
\begin{theorem}[A moderate deviations principle]\label{MD}
Let Assumptions  \ref{Ass Ap}, \ref{Ass Exp} and \ref{AsCov} hold and 
suppose that $\Sig^2>0$.
Let $\ka_\om$ be either $\mu_\om$ or $\bar m_\om$. Let $a_n$ be a sequence of positive numbers so that 
\[
\lim_{n\to\infty}\frac{a_n}{\sqrt n}=\infty\, \text{ and }\, \lim_{n\to\infty}\frac{a_n}{n}=0
\]
and set $\ve_n=n/a_n^2$. 
In both cases we set $W_n=W_n^\om=\big(S_n^\om\varphi-\ka_\om(S_n^\om\varphi)\big)/a_n$.
Then for $\bbP$-a.e. $\om$, for any Borel measurable set $\Gam\subset\bbR$ we have
\[
-\inf_{x\in\Gamma^o}I_0(x)\leq \liminf_{n\to\infty}\ve_n\ln\ka_\om(W_n^\om\in\Gamma)\leq \limsup_{n\to\infty}\ve_n\ln\ka_\om(W_n^\om\in\Gamma)\leq-\inf_{x\in\bar\Gamma}I_0(x)
\]
where $I_0(x)=\frac12 x^2/\Sig^2$, $\Gamma^o$ is the interior of $\Gamma$ and $\bar\Gam$ is its closure.
\end{theorem}

We also get the following local large deviations principle
\begin{theorem}[Local large deviations principle]\label{LD}
Let Assumptions  \ref{Ass Ap}, \ref{Ass Exp} and \ref{AsCov} hold and 
suppose that $\Sig^2>0$.
Let $\ka_\om$ be either $\mu_\om$ or $\bar m_\om$. 
In both cases we set $A_n=A_n^\om=(S_n^\om\varphi-\ka_\om(S_n^\om\varphi))/n$. Then there exists a constant $\del>0$ so that $\bbP$-a.s.  for any Borel measurable set $\Gam\subset[-\del,\del]$ we have
\[
-\inf_{x\in\Gamma^o}I(x)\leq \liminf_{n\to\infty}\frac1n\ln\ka_\om(A_n^\om\in\Gamma)\leq \limsup_{n\to\infty}\frac1n\ln\ka_\om(A_n^\om\in\Gamma)\leq-\inf_{x\in\bar\Gamma}I(x)
\]
where $I$ is the Fenchel-Legendre transform of the average pressure function $\cP(t)=\int \ln \la_\om(t)dP(\om)$. Moreover, for every $\ve>0$ small enough
\[
\lim_{n\to\infty}\frac 1n\ln \ka_\om(S_n^\om \varphi-\ka_\om(S_n^\om \varphi)\geq \ve n)=-I(\ve).
\]
\end{theorem}

\section{Applications}\label{SecApp}
\subsection{limit theorems for non-uniformly random expanding systems}
We consider here a direct random generalization of the model considered by Melbourne and Nicol \cite{MN}.  
Suppose  there are constants $\la>1$,  $\eta\in(0,1)$, $C\geq1$, $c_1,c_2,c_3>0$ so that 
\vskip0.2cm

(i) $M_\om=(M_\om,\rho_\om)$ is a bounded locally compact metric space and $f_{\om}^{R_{\om,j}}$ is a measurable bijection between $\Lambda_{\om,j}$ and $\Del_{\sig^{R_{\om,j}}\om,0}$.
\vskip0.1cm
(ii) $\rho_{\sig^{R_{\om,j}}\om}(f_\om x^{R_{\om,j}},f_\om y^{R_{\om,j}})\geq\la \rho_\om(x,y)$ for all $j$ and  $x,y\in\Del_{\om,j}$;
\vskip0.1cm
(iii) $\rho_{\sig^\ell\om}(f^\ell_\om x,f^\ell_\om y)\leq C\rho_{\sig^{R_{\om,j}}\om}(f_\om^{R_{\om,j}}x,f_\om^{R_{\om,j}}y)$ for all $j$, $x,y\in\Lambda_{\om,j}$ and $\ell<R_{\om,j}$;
\vskip0.1cm
(iv) The functions $g_{\om,j}=\frac{d(f_\om^{R_\om,j})_*(m_\om|\Lambda_{\om,j})}{dm_\om|\Del_{\sig^{R_{\om,j}}\om,0}}$ satisfy
\[
\left|\log g_{\om,j}(x)-\log g_{\om,j}(y)\right|\leq C\rho_\om(x,y)^\eta
\]
for any $x,y\in\Del_{\om,0}$;
\vskip0.1cm
(v) For $\bbP$ a.e. $\om$ we have $m_\om(R_{\om}\geq n)\leq c_1e^{-c_2n}$ for every $n$;
\vskip0.1cm
(vi) There are $N_0$ and $t_1,t_2,...,t_{N_0}\in\bbN$ such that $\gcd\{t_i\}=1$ and $\bbP$-a.s. 
$m_\om(R_{\om}=t_i)>0$;  Moreover, $R_\om$ is a stopping time, namely the map $(\om,x)\to R_{\om}(x)$ is measurable and if $R_{\om}(x)=n$ then also $R_{\om'}(x)=n$, where $\om'$ is a sequence whose first $n$ coordinates are the same as $\om$'s;
\vskip0.1cm
\vskip0.2cm 

The first four assumptions are straight forward generalizations of classical deterministic assumptions, and they  mean that the maps $f_\om$ are a random family of non-uniformly distance expanding maps, while  the sixth assumption comes from \cite{BBD} (see also \cite{ABR} and \cite{Du}). Under these assumptions, the map $\pi_\om:\Del_\om\to M_\om$ given by $\pi_\om(x,\ell)=f_{\sig^{-\ell}\om}^{\ell}x$ is a Holder continuous bijection on its image.

We consider now a uniformly bounded family of H\"older continuous functions $\varphi_\om:M_\om\to\bbR$ (uniformly in $\om$) and define 
\[
S_n^\om \varphi=\sum_{j=0}^{n-1}\varphi_{\sig^j\om}\circ f_\om^n.
\]
For a fixed $\om$ we will view $S_n^\om \varphi$ as a sequence of random variables with respect to either $(\pi_\om)_*\mu_\om$, which is an equivariant family of measures equivalent to the restriction of the reference measures $m_\om$ to the image of $\pi_\om$ (``sample stationary measures" in the terminology of \cite{BBD}) or the measure $(\pi_\om)_*\textbf{m}_\om$ (which is also equivalent to the latter restriction, and coincides with $m_\om$ on the random base $\Del_{\om,0}$).
In order for our results in Section \ref{Sec1} to hold we need  Assumption \ref{AsCov} to hold true.  Using Remark \ref{RemApp}, we have the following.
\begin{proposition}\label{PropApp}
For the maps describe above, Assumption \ref{AsCov} holds true on the random tower if 
one of the following two conditions hold true.

(i) For any $i$ we have 
\[
\text{ess-sup }\sup_{x\in\Del_{\om,i}}|J f_\om^{R_{\om,i}}x|<\infty
\]
(equivalently the Jacobian of $f_\om^{R_\om}$ restricted to the atom with the $i$-th largest measure is uniformly bounded in $\om$).

(ii) There is  a constant $C>0$ so that, $\bbP$-a.s. we have $|Jf_\om|\leq C$. Moreover, for all $n$ large enough there is a constant $k_n$ so that $\bbP$-a.s. the set $\{i: n\leq R_{\om,i}\leq n+k_n\}$ is non-empty.
\end{proposition}

\subsection{Limit theorems for random nonuniformly hyperbolic maps}
Let $M$ be  a smooth compact Riemannian manifold and $f\in\text{Diff}^{1+}(M)$ have a transitive partially hyperbolic set $K\subset M$ and a local unstable manifold $D\subset K$. As in \cite{ABR}, let $\cF$ be a sufficiently small $C^1$-ball around $f$. Let $P_0$ be a probability measure on $\cF$ with a compact support $\mathbb B$. Furthermore, let $(\Om_0,\cF_0,P_0)$ be a probability space and $f_{\om_0},\,\om_0\in\Om_0$ be a random $\mathbb B$-valued element. 
We then consider $f_\om=f_{\om_0}$, where $\om=\{\om_n\}\in\Om=\Om_0^\bbZ$. As in  \cite{ABR}, we will also assume that $f_{\om_0}$ is $C^1$-close to $f|_{D}$ on domains $\{D_{\om_0}\}$  of $cu$-nonuniform expansions (see the exact definition after (10) in \cite{ABR}).

We claim that our results hold true  for the above partially hyperbolic maps, together with the physical measures $\mu_\om$ from \cite[Theorem 1.5]{ABR}.
Indeed, we first observe that the random towers constructed there have exponential tails uniformly in $\om$. Moreover, relying on \cite[Propositions 3.3]{ABR} and \cite[Proposition 3.5]{ABR} (which are random versions of \cite[Lemma 4.4]{ADLP}) and arguing as in \cite[Section 7]{ADLP} one can show that, after collapsing along stable manifolds we get a H\"older continuous random conjugacy  with a random Gibbs-Markov-Young map, a model which can be reduced to the random towers considered in this paper (this essentially means that the arguments in \cite{ABR} reduce the problem to random towers so that (\ref{Distortion0}) holds true  for some $\be$ with our separation time and not only with the (smaller) random separation time defined in \cite{ABR}).  We also note that, in view of (76) in \cite{ABR}, the condition that $\{i: \ell\leq R_{\om,i}\leq \ell+k_\ell\}$ is non-empty holds true with $k_\ell=L$ which does not depend on $\ell$. Therefore, as discussed in Remark \ref{RemApp} we get that the conditions in Assumption \ref{AsCov}  are valid. Finally, we note that we indeed get all the limit theorems  for the original maps $f_\om$ from the results on the random tower because  (7) in \cite{ABR} hold true with $\del_{\sig^k\om,k}=C\del^k$ for some $C>0$ and $\del\in(0,1)$ (using that, the reduction from the invertible case to the non-invertible case is done similarly to \cite[Section 4.2.2]{LLT2020}).

\section{Random transfer operators}\label{Sec4}
In this section we  obtain several abstract results on random towers. We  start from results which hold true when the tails decay  sub-exponentially fast, and the exponential rate of decay will only be used in Section \ref{C-sec} when dealing with complex cones. 

In what  follows we will constantly use the following simple result.
\begin{lemma}\label{CorDist}
There exists a constant $Q>0$ so that for all $\om$, $k$ and $x\in\Del_\om$ such that $F_\om^jx\in\Del_{\sig^j\om,0}$ for some $1\leq j\leq k$
we have
\[
Q^{-1}m_\om(C_{\om,k}(x))\leq \frac1{JF_\om^k x}\leq Q m_\om(C_{\om,k}(x)).
\]
\end{lemma}

\begin{proof}
First, iterating (\ref{Distortion}), we get that for some $C_1>0$ and  all $n\geq1$ and $x,y$ which belong to the same $n$-th length cylinder we have
\begin{equation}\label{Distortion}
\left|\frac{JF_\om^n x}{JF_\om^n y}-1\right|\leq C_1\be^{s_{\sig^n\om}(F_\om^n x,F_\om^n y)}.
\end{equation}

Next, in order to prove \eqref{Distortion} let us assume first that $F_\om^kx\in\Del_{\sig^k\om,0}$. Then the map $F_\om^k|_{C_{\om,k}(x)}$ is injective and its image is $\Del_{\sigma^k \om, 0}$. Let $g_k:\Del_{\sig^k\om,0}\to C_{\om,k}(x)$ be the corresponding inverse branch. Then the lemma follows from (\ref{Distortion}) together with the equality
\[
m_\om(C_{\om,k}(x))=\int_{\Del_{\sig^k\om, 0}}Jg_k dm_{\sig^k\om}.
\] 
In the general case, let $j_0\leq k$ be the maximal index so that $F_\om^{j_0}x\in\Del_{\sig^{j_0}\om,0}$. Then
\[
C_{\om,k}(x)=C_{\om,j_0}(x)\,\text{ and }\, JF_\om^kx=JF_\om^{j_0}x
\]
which reduces the problem to the case when $j_0=k$.
\end{proof}
\begin{remark}\label{Rem}
If $F_\om^jx\notin\Del_{\sig^j\om,0}$ for all $1\leq j\leq k$ then $C_{\om,k}(x)=\Del_{\om,\ell,i}=C_{\om,r}(x)$, where $r$ is the first time that $F_\om^r x\in\Del_{\sig^r\om,0}$ and $\Del_{\om,\ell,i}$ is the atom containing $x$. Therefore, 
\[
Q^{-1}m_\om(C_{\om,k}(x))\leq \frac1{JF_\om^{r} x}=\frac1{J f_{\sigma^{-\ell}\om}^{R_{\sigma^{-\ell}\om}}x}\leq Q m_\om(C_{\om,k}(x)) 
\]
where $x=(x_0,\ell)$. We conclude that for any cylinder $C_{\om,k}$ and any point $x=(x_0,\ell)\in C_{\om,k}$ we have
\[
Q^{-1}m_\om(C_{\om,k})\leq \frac1{J (f^R)_{\sig^{-\ell}\om}^s x_0}\leq Q m_\om(C_{\om,k}) 
\]
where $s$ is the number of $j$'s between $1$ and $k$ so that $F_\om^j x\in\Del_{\sig^k\om,0}$.
\end{remark}

\subsection{Random complex transfer operators}
Let $\varphi_\om:\Del_\om\to\bbR$ be a H\"older continuous function with respect to the metric 
\[
d_\om(x,y)=\be^{s_\om(x,y)}
\] 
so that $(\om,x)\to\varphi_\om(x)$ is a measurable map. 
For every $n\geq1$ we consider the random function
\[
S_n^\om\varphi=\sum_{j=0}^{n-1}\varphi_{\sig^j\om}\circ F_\om^j.
\]
Since $F_\om$ is at most countable to one, 
for any complex number $z$ we can define the transfer operator $P_{\om}^{z}$ by 
\[
P_\om^{z}g(x)=\sum_{y: F_\om y=x}\frac1{JF_\om(y)}e^{z\varphi_\om(y)}g(y),
\]
where $g:\Del_\om\to\bbC$ and $x\in\Del_{\sig \om}$. This operator takes a function on $\Del_\om$ and returns a function on $\Del_{\sig\om}$. Let us also consider the iterates of these operators 
\[
P_\om^{z,n}=P_{\sig^{n-1}\om}^{z}\circ\cdots \circ P_{\sig\om}^{z}\circ P_\om^{z}.
\]
Then 
\[
P_\om^{z,n}g(x)=\sum_{y: F_\om^n y=x}\frac1{JF_\om^n(y)}e^{zS_n^\om\varphi(y)}g(y).
\]

\subsubsection{Weighted norm spaces}
Let $(v_\ell)_{\ell=0}^\infty$ be a monotone increasing strictly positive sequence so that for $\bbP$-a.e. $\om\in\Om$,
\begin{equation}\label{A1'}
\sum_{\ell=0}^\infty v_\ell m_{\sig^{-\ell}_\om}(\{x_0: R_{\sig^{-\ell}\om}(x_0)\geq \ell)\leq C_2
\end{equation}
for some $C_2>0$ not depending on $\om$. Later on we will assume the uniform exponential tails assumption (\ref{Exp}), and then we will take $v_\ell=c_1e^{c\ell}$ for some $c_1,c>0$, but for the meanwhile we will obtain our results for general sequences $(v_\ell)$, since we think it is interesting on its own.
We define a norm on functions $g:\Del_\om\to\bbC$ as follows:
\[
\|g\|_\om=\|g\|_s+\|g\|_h
\]
where 
\[
\|g\|_{s}=\sup_\ell v_\ell^{-1}\|g\bbI_{\Del_{\om,\ell}}\|_\infty, \,\|g\|_h=\sup_\ell v_\ell^{-1}\big|g\big|_{\om,\Del_{\om,\ell}}
\]
where for any $A\subset \Del_\om$, 
\begin{equation}\label{g norm def}
|g|_{\om,A}=|g|_{\om,A,\be}=\sup_{ x,y\in A\,\,x\not=y} \frac{|g(x)-g(y)|}{d_\om(x, y)}
\end{equation}
(the dependence on $\beta$ is through the metric $d_\om$).
Note that
\begin{equation}\label{L q g s}
\|g\|_{L^1(m_\om)}\leq C_2\|g\|_s
\end{equation}
for every function g $g$. Indeed,
\begin{eqnarray*}
\|g\|_{L^1(m_\om)}=\sum_{\ell\geq0}\int_{\Del_{\om,\ell}}|g|dm_\om\\\leq \|g\|_s\sum_{\ell}v_\ell m_\om(\Del_{\om,\ell})=\|g\|_s\sum_{\ell=0}^\infty v_\ell m_{\sig^{-\ell}_\om}(x_0: R_{\sig^{-\ell}\om}(x_0)\geq \ell).
\end{eqnarray*}

\subsubsection{A Lasota-Yorke inequality}
We will prove here the following results.

\begin{proposition}\label{Prop LY}
(i) For every $N$ and $\ell$ so that $N\leq \ell$, a function $g:\Del_{\om}\to\bbC$ and $x,y\in\Del_{\sig^N\om,\ell}$ we have 
\begin{equation}\label{LY1.1}
|P_\om^{it,N}g(x)|\leq v_{\ell-N}\|g\|_s
\end{equation}
and 
\begin{equation}\label{LY1.2}
|P_\om^{it,N}g(x)-P_\om^{it,N}g(y)|\leq (\|g\|_h\beta^N+(A|t|+2\be^{-1})\|g\|_s)v_{\ell-N}d_{\sig^N\om}(x,y)
\end{equation}
where $A=(1-\be)^{-1}\text{ess-sup }\sup_{\ell}|\varphi_\om|_{\om,\Del_{\om,\ell}}<\infty$.

(ii) For all $N$ and $\ell$ so that $N>\ell$, a function $g:\Del_{\om,\ell}\to\bbC$ and $x,y\in\Del_{\sig^N\om,\ell}$ we have 
\begin{equation}\label{LY2.1}
|P_\om^{it,N}g(x)| \leq Q\left(\int |g|d m_\om+\be^N\|g\|_h\cdot C_2\right):=R_{N}(g)
\end{equation}
and 
\begin{equation}\label{LY2.2}
|P_\om^{it,N}g(x)-P_\om^{it,N}g(y)|\leq \left(C_1+2\be^{-1}+|t|A\right)R_N(g)d_{\sig^N\om}(x,y)
\end{equation}
where $C_1$ comes from \eqref{Distortion} and $C_2$ comes from (\ref{A1'}).

In particular 
\begin{eqnarray*}
\|P_\om^{it,N}g\|_{\sig^N\om}\\\leq \max\left(\sup_{\ell\geq N}v_{\ell-N}v_\ell^{-1}\left((1+|A|t)\|g\|_s+\beta^N\|g\|_h\right), v_0^{-1}R_N(g)(2+C_1+|t|A)\right).
\end{eqnarray*}
Therefore, for any compact sets $J\subset\bbR$ the operator norms $\|P_\om^{it,N}\|_{\om,\sig^N\om}$ with respect to the norms $\|\cdot\|_\om$ and $\|\cdot\|_{\sig^M\om}$ are uniformly bounded in $\om\in\Om, N\geq1$ and $t\in J$.
\end{proposition}

\begin{proof}
Let $g:\Del_\om\to\bbC$ and $\ell,N\geq1$. We assume first that 
 $N\leq\ell$. Then for any  $(x,\ell)\in\Del_{\sig^N\om, \ell}$ we have
\[
|P_\om^{it,N}g(x,\ell)|=|g(x,\ell-N)e^{it S_N^\om\varphi(x,\ell-N)}|\leq v_{\ell-N}\|g\|_s,
\]
which yields~\eqref{LY1.1}.
Moreover, for any $x_\ell=(x,\ell), y_\ell=(y,\ell)\in\Del_{\sig^N\om, \ell}$ that belong to the same partition element,  we have that  
\begin{eqnarray*}
|P_\om^{it,N}g(x_\ell)-P_\om^{it,N}g(y_\ell)|=|e^{it S_N^\om \varphi(x,\ell-N)}g(x_{\ell-N})-e^{it S_N^\om \varphi(y, \ell-N)}g(y_{\ell-N})|\\\leq |g(x_{\ell-N})-g(y_{\ell-N})|+\\|t|v_{\ell-N}\|g\|_s\sum_{j=0}^{N-1}|\varphi_{\sig^j\om}(x,\ell-N+j)-\varphi_{\sig^j\om}(y,\ell-N+j)|:=I_1+I_2.
\end{eqnarray*}
Since $d_{\om}(x_{\ell-N},y_{\ell-N})=\be^N d_{\sig^N\om}(x_\ell,y_\ell)$ we have
\[
I_1\leq v_{\ell-N}\|g\|_h\be^N d_{\sig^N\om}(x_\ell,y_\ell).
\]
Similarly, with $|\varphi_\om|:=\sup_{\ell}|\varphi_\om|_{\om,\Del_{\om,\ell}}$, where the last semi-norm is defined in \eqref{g norm def}, we have
\begin{eqnarray*}
\sum_{j=0}^{N-1}|\varphi_{\sig^j\om}(x,\ell-N+j)-\varphi_{\sig^j\om}(y,\ell-N+j)|\\\leq d_{\sig^N\om}(x_\ell,y_\ell)\text{ess-sup}|\varphi_\om|(\be^N+\be^{N-1}+...+\be^{N-j}).
\end{eqnarray*}
By combining the above estimates, we conclude that~\eqref{LY1.2} holds. 

Let us now consider the case when $x_\ell$ and $y_\ell$ do not belong to the same partition element. In this case, we have that 
\[
\begin{split}
| P_\om^{it, N}g(x_\ell)-P_\om^{it, N}g(y_\ell)| &\le |P_\om^{it, N}g(x_\ell)|+|P_\om^{it, N}g(y_\ell)|\\
&=|g(x, \ell-N)|+|g(y, \ell-N)| \\
&\le 2 v_{\ell-N}\lVert g\rVert_s \\
&=2v_{\ell-N}\be^{-1}\lVert g\rVert_s d_{\sig^N \om}(x_\ell, y_\ell),
\end{split}
\]
where in the last equality we have used that $d_{\sig^N \om}(x_\ell, y_\ell)=\beta$ since the separation time of their orbit is $1$. We conclude that~\eqref{LY1.2} also holds in the above case.

Now we will prove the second item. Suppose that $\ell<N$, and let $(x,\ell)=x_\ell\in\Del_{\sig^N\om, \ell}$. 
For any cylinder $C_N$ of length $N$ in $\Del_\om$ the map $F_\om^N|_{C_N}$ is surjective, and it defines an inverse branch of $F_\om^N$ (onto its image). Let use denote by $x_N=x_N(C_N)$ the unique preimage of $x_\ell$ under $F_\om^N$ which belongs to $C_N=C_N(x_N)$ (if such a preimage exists).
We then have
\begin{equation}\label{I3}
|P_\om^{it,N}g(x,\ell)|\leq\sum_{C_N}\left|\frac1{JF^N_\om(x_N)}\right|\cdot|g(x_N)|
\end{equation}
where the sum is over all cylinders $C_N$ for each $x_N(C_N)$ exists. 
Fix some cylinder $C_N$ and set
\[
A_g(C_N)=\frac1{m_\om(C_N)}\int_{C_N}gdm_\om.
\]
Then,
\[
|g(x_N)| \leq |A_g(C_N)|+\sup_{y_1,y_2\in C_N}|g(y_1)-g(y_2)|.
\]
Next, by Lemma \ref{CorDist} for any cylinder $C_N$ we have
\[
\left|\frac1{JF^N_\om(x_N)}\right|\leq Q m_\om(C_N).
\]
Note that we can indeed apply Lemma \ref{CorDist} since $\ell<N$ and so $F_\om^{N-\ell}x_N$ belongs to the $0$-th floor.
Since the diameter of $C_N$ does not exceed $\be^N$, we conclude that
 \begin{eqnarray}\label{I2}
|P_\om^{it,N}g(x,\ell)|\leq Q\int |g|dm_\om\\+Q\sum_{C_N}
\be^N\sum_{k\geq0}\sum_{C_N\subset \Del_{\om,k}}m_\om(C_N)|g|_{\be,\Del_{\om,k}}\nonumber\\\leq
Q\int |g|dm_\om
+\be^NQ\sum_{k\geq0}\sum_{C_N\subset \Del_{\om,k}}v_k m_\om(C_N)v_k^{-1}|g|_{\be,\Del_{\om,k}}\nonumber\\\leq Q\left(\int |g|d m_\om+\be^N\|g\|_h\cdot\sum_{k\geq0}v_km_\om(\Del_{\om,k})\right)\nonumber,
 \end{eqnarray}
and the proof of  (\ref{LY2.1}) is completed.

Now we will prove (\ref{LY2.2}). Let $x_\ell=(x,\ell)$ and $y_\ell=(y,\ell)$ belong to $\Del_{\sig^N\om,\ell}$. When  they do not belong to the same partition element on the $\ell$-th floor then $d_{\sig^N\om}(x_\ell,y_\ell)=\be$, and so (\ref{LY2.2}) follows from (\ref{LY2.1}). Suppose now that $d_{\sig^N\om}(x_\ell,y_\ell)<\be$.
Then we can couple the inverse images of $x_\ell$ and $y_\ell$ under $F_\om^N$ and index them according to a subset of cylinders of length $N$, so that the preimage indexed by $C_N$ belongs to $C_N$. 
That is, the preimgaes $\{x'(C_N)\}$ and $\{y'(C_N)\}$ have the form
\[
x'=x'(C_N)=\left(F_\om^N|_{C_N}\right)^{-1}x_\ell \text{ and } y'=y'(C_N)=\left(F_\om^N|_{C_N}\right)^{-1}y_\ell.
\]
We have
\begin{eqnarray*}
|P_\om^{it,N}g(x_\ell)-P_\om^{it,N}g(y_\ell)|\leq\sum_{C_N}\left|\frac1{JF^N_\om x'}e^{it S_N^\om\varphi (x')}g(x')-\frac1{JF^N_\om y'}e^{it S_N^\om\varphi y'}g(y')\right|.
\end{eqnarray*}
Fix some $C_N$ and $x'=x'(C_N)$ and $y'=y'(C_N)$. We also set $g_{N,t}=e^{it S_N^\om\varphi}g$.
Then
\begin{eqnarray*}
\left|\frac1{JF^N_\om x'}e^{it S_N^\om\varphi (x')}g(x')-\frac1{JF^N_\om y'}e^{it S_N^\om\varphi (y')}g(y')\right|\\\leq \frac{|g_{N,t}(x')-g_{N,t}(y')|}{JF^N_\om x'}+|g(y')|\left|\frac1{JF^N_\om x'}-\frac1{{JF^N_\om y'}}\right|\\\leq
\frac{|g(x')|\cdot|e^{it S_N^\om\varphi(x')}-e^{it S_N^\om \varphi(y')}|}{|JF^{N}_\om x'|}
+\frac{|g(x')-g(y')|}{|JF_\om^{N} x'|}\\+|g(y')|\cdot
\left|\frac1{JF_\om^{N} x'}-\frac1{JF_\om^{N} y'}\right|:=I_1+I_2+I_3.
\end{eqnarray*}
By the distortion assumption (\ref{Distortion}) on $J F_\om$ we have
\[
I_3\leq C_1|g(y')|\be^{s_{\sig^N\om}(x_\ell,y_\ell)}/ |JF_\om^N y'|.
\]
Therefore, the contribution to the sum over $C_N$ coming from $I_3$ is bounded from above by the right hand side of (\ref{I3}) times $C_1\be^{s_{\sig^N\om}(x_\ell,y_\ell)}$. Moreover, also the contribution coming from $I_2$ does not exceed the right hand side of (\ref{I2}) multiplied by $\be^{s_{\sig^N\om}(x_\ell,y_\ell)}$. It remains to estimate $I_1$. Using the mean value theorem and that $\varphi_\om$ are uniformly H\"older continuous we have 
\begin{eqnarray*}
|e^{it S_N^\om \varphi(x')}-e^{it S_N^\om \varphi(y')}|\leq|t|\sum_{k=0}^{N-1}|\varphi_{\sig^k\om}(F_\om^{k} x')-\varphi_{\sig^k\om}(F_\om^{k} y')|\\\leq\|\varphi\||t|\sum_{k=0}^{N-1}\be^{s_{\sig^k\om}(F_\om^{k} x',F_\om^{k} y')}=
\|\varphi\||t|\be^{s_{\sig^N\om}(x_\ell,y_\ell)}\sum_{k=0}^{N-1}\be^{k} \\
\leq A|t|\be^{s_{\sig^N\om}(x_\ell,y_\ell)}
\end{eqnarray*}
where $\|\varphi\|:=\text{ess-sup }\sup_{\ell}|\varphi_\om|_{\Del_{\om,\ell}}$.
This completes the proof of the proposition.
\end{proof}

\subsubsection{Application: the $\al$-mixing condition}
The following corollary will play an important role in the proof that the cylinders are $\al$-mixing. In the deterministic case this result was (essentially) proven in \cite[Lemma 4]{HydPsyl}, but we will provide a different proof. We consider the following norm of a function $g_\om \colon \Delta_\om
\to\bbC$ 
\[
\|g\|_{Li}=\|g\|_{Li,\om}=\|g\|_\infty+|g|_\om
\]
where $\|g\|_\infty=\sup|g|$ and 
\begin{equation}\label{g beta}
|g|_\om=|g|_{\om,\be}=\sup_{\ell\geq0}|g|_{\om,\Del_{\om,\ell}}.
\end{equation}
Then $\|g\|_{Li,\om}=\|gv\|_\om=\|gv\|_s+\|gv\|_h$ for any $g:\Del_\om\to\bbC$, where $gv(x,\ell)=v_\ell g(x)$. Let us also define $\cH_{\om}=\cH_{\om,\be}$ to be the linear space of all functions  $g_\om \colon \Delta_\om\to\bbC$ so that $\|g\|_{Li,\om}<\infty$. Then $\cH_{\om}$ is a Banach space.

\begin{corollary}\label{CorLY}
There exists a constant $C_3>0$ so that for $\bbP$-a.e. $\om$, 
$g:\Del_\om\to\bbC$, $N\geq1$ and a function $u:\Del_\om\to\bbC$ which is constant on cylinders of order $N$, 
\[
\|P_\om^{0,N}(gu)\|_{Li,\te^N\om}\leq C_3\left(1+(\sup|g|+\sup|u|)^2+|g|_\om\right).
\]
\end{corollary}
\begin{proof}
Let $(x,\ell),(y,\ell)\in\Del_{\om,\ell}$. Assume first that $N\leq \ell$. It is clear that 
\[
|P^{0,N}_\om (gu)(x,\ell)|=|g(x,\ell-N)u(x,\ell-N)|\leq \sup|g|\sup|u|.
\]
Next, observe that $|u|_\om\leq \sup2|u|\be^{-N}$ (since $u(x)=u(y)$ if $d_\om(x,y)\leq\be^N$). Therefore,
\begin{eqnarray*}
|P^{0,N}_\om (gu)(x,\ell)-P^{0,N}_\om (gu)(y,\ell)|\\=|g(x,\ell-N)u(x,\ell-N)-g(y,\ell-N)u(y,\ell-N)|\\\leq 
\sup|g|\cdot |u(x,\ell-N)-u(y,\ell-N)|+\sup|u||g|_\om \beta^N d(x,y)\leq \\2\sup|g|\sup|u|\beta^Nd(x,y)\beta^{-N}+\sup|u||g|_\om \beta^N d(x,y)\\=(2\sup|g|+\beta^N|g|_\om)\sup|u|d(x,y).
\end{eqnarray*}
The desired estimates in the case $N>\ell$ follow from Proposition \ref{Prop LY} (ii) applied with the function $gu$.
\end{proof}

Next, define
\begin{equation}\label{d k}
d_{k}=\text{ess-sup}_\om\sup_{g\in\cH_{+,\om}}\|P_\om^{0,k} g-m_\om(g)h_{\sig^k\om}\|_{L^1(m_{\sig^k\om})}/\|g\|_{Li}.
\end{equation}
Here  $\cH_{+,\om}$ is the space of all non-negative functions on $\Del_\om$ so 
 that $\|g\|_{Li,\om}<\infty$ (note\footnote{Here $gdm_\om$ denotes the absolutely continuous measure w.r.t. $m_\om$ whose density is $g$.} that $\|P_\om^{0,k} g-m_\om(g)h_{\sig^k\om}\|_{L^1(m_{\sig^k\om})}
 =\|(F_\om^k)_*(gdm_\om)-\mu_{\sig^k\om}\|_{TV}$, and that it is enough to consider $g$'s so that $m_\om(g_\om)=1$). 
The following result  is a particular case of \cite[Theorem 2.5]{ABR}.
\begin{theorem}\cite[Theorem 2.5]{ABR}\label{AssMixing}
If $m_\om(R_\om\geq k)$ decay (stretched) exponentially fast to $0$ uniformly in $\om$ then 
$d_k$ decays (stretched) exponentially fast to $0$. If $m_\om(R_\om\geq k)\leq Ck^{-a-1}$ for some $a>1$ then 
$d_k=O(k^{-(a-1-\ve)})$ for every $\ve>0$.
\end{theorem}

Now we are ready to prove the aforementioned $\alpha$-mixing results. Let $\cA_{\om,n}$ be the $\sigma$-algebra generated by all the cylinder sets $C_{\om,n}$ of order $n$ in $\Del_\om$.

\begin{proposition}\label{MixProp}
There is a constant $D>0$ so that for any $\om,n,k\geq0$, $A\in \cA_{\om,n}$ and a measurable set $B\subset\Del_{\sig^{n+k}\om}$,
\begin{equation}\label{al-mix}
\left|\mu_\om(A\cap (F_\om^{n+k})^{-1}B)-\mu_\om(A)\mu_\om((F_\om^{n+k})^{-1}B)\right|\leq Dd_k.
\end{equation}
\end{proposition}

\begin{proof}
The proof of (\ref{al-mix}) continuous similarly to \cite[Section 4.1]{HydPsyl}. That is, using that $P_\om$ is the dual of $F_{\sig\om}$
 we get that
 \begin{eqnarray}\label{alBase}
 \mu_\om(A\cap (F_\om^{n+k})^{-1}B)-\mu_\om(A)\mu_\om((F_\om^{n+k})^{-1}B)\\=\int_B\left(P^{0,k}_{\sig^n\om}(\zeta)-\mu_\om(A)h_{\sig^{n+k}\om}\right)dm_{\sig^{n+k}\om}\nonumber
 \end{eqnarray}
where $\zeta=P_{\om}^{0,n}(\bbI_{A}h_\om)$. By Corollary \ref{CorLY} we have $\|\zeta\|_{Li}\leq C_3$. This clearly yields (\ref{al-mix}), taking into account that 
\[
m_{\sig^n\om}(\zeta)=m_\om(\bbI_{A}h_\om)=\mu_\om(A).
\]
\end{proof}

\subsection{Random partitions}
We  define a new measure on $\Del_\om$ by  $\tilde m_\om=vd m_\om$, where $(v_\ell)$ is the sequence from the previous section. Our assumption here concerning these measure is that 
\begin{equation}\label{UnfT}
\lim_{\ell\to\infty}\text{ess-sup}_\om \tilde m_\om(\cup_{j\geq\ell}\Del_{\om,j})=0.
\end{equation}
In  Section \ref{SecLimThms} we will have stronger assumptions on the rate of decay of $m_\om(R_{\om}\geq n)$, but we believe that the partitions constructed here  have their own interest, and so the results are formulated under weaker conditions (and for general increasing sequences $(v_\ell)_{\ell\geq 0}$).

We first need the following result.

\begin{proposition}\label{PrpCn1}
Under (\ref{UnfT}) and Assumption \ref{AsCov}, for every $\ve>0$ and $s\in\bbN$ there are $\del>0$, $M\geq1$ so that for $\bbP$-a.a. $\om$ there are at most $M$ disjoint cylinders $A_{\om,1},...,A_{\om,j_\om}$, $j_\om\leq M$ of order  $s$ in $\Del_\om$ so that for all $1\leq i\leq M$,
\begin{equation}\label{MIN}
\min \{ \mu_\om(A_{\om,i}),m_\om(A_{\om,i})\}\geq\del
\end{equation}
and with $A_{\om,j_\om+1}=\Del_\om\setminus(A_{\om,1}\cup\cdots\cup A_{\om,j_\om})$ we have
\[
\del\leq \min \{\mu_\om(A_{\om,j_\om+1}),m_\om(A_{\om,j_\om+1})\}\,\text{ and }\,\tilde m_\om(A_{\om,j_\om+1})<\ve.
\]
\end{proposition}

\begin{proof}
Let $\ve>0$ and $s\in\bbN$ and fix some $\om$. Let $\ve'>0$ (which is yet to be determined), and  $Q_{\sig^j\om,1},...,Q_{\sig^j\om,k_{\sig^j\om}}$, $k_{\sig^j\om}\leq J$ be at most $J$ atoms on $\Del_{\sig^j\om}$ (for $0\leq j<s$), so that 
\[
 m_{\sig^j\om}\left(\Del_{\sig^j\om}\setminus(Q_{\sig^j\om,1}\cup Q_{\sig^j\om,2}\cup\cdots\cup Q_{\sig^j\om,k_{\sig^j\om}})\right)<\ve'
\]
and the $m_{\sig^j\om}$-measure of each $Q_{\sig^j\om,k}$ and of the complement of their union is not less than $\del'$ for some $J$ and $\del'>0$ which depend only on $\ve'$.
We define $A_{\om,1},...,A_{\om,j_\om}$ to be the nonempty cylinders among the cylinder of order $s$  of the form
\[
\bigcap_{i=0}^{s-1}(F_\om^i)^{-1}Q_{\sig^i\om,u_i}
\]
where $u_0,...,u_{s-1}$ are so that $u_i\leq k_{\sig^i\om}$ (note that $j_\om\leq J^s=M$). Set $B=B_\om=\Del_\om\setminus(A_{\om,1}\cup\cdots\cup A_{\om,j_\om})$. Using Lemma \ref{CorDist} and Remark \ref{Rem} we obtain that for each $u_0,...,u_{s-1}$ as above we have
\begin{eqnarray*}
m_\om\left(\bigcap_{i=0}^{s-1}(F_\om^i)^{-1}Q_{\sig^i\om,u_i}\right)\geq \frac{Q^{-1}}{F_\om^{s}x}=
\frac{Q^{-1}}{(f_{\sig^{-\ell}\om}^R)^{s'}x_0}\\\geq Q^{-1-s'}m_{\om}(Q_{\om,u_0})\prod_{j=1}^{s'-1}m_{\sig^{v_j-\ell}\om}(\cA_{\sig^{v_j-\ell}\om}(f_{\sig^{-\ell+v_{j-1}}\om}^{v_j}x_0))
\geq Q^{-s-1}(\del')^s.
\end{eqnarray*}
Here $x=(x_0,\ell)$ is an arbitrary point in the cylinder under consideration, $\ell=\ell_{\om,u_0,...,u
s-1}$ is the level of the cylinder,  $s'\leq s-1$ is the number of returns to the base, $v_0=0$, $v_j=v_{i,\om,u_0,...,u_{s-1}}$, $1\leq j\leq s'$ are the times these returns occur, $\cA_{\om}(y)$ is the atom in $M_\om$ containing $y$  
and we have used that each return happens after the orbit of $x$ visits one the atoms $Q_{\sig^i\om,u_i}$.
 Note that in the above arguments we formally assume that $F_\om^s x$ belongs to $\Del_{\sig^s\om,0}$ for any $x$ in the above cylinder. This is not really a restriction since otherwise we could have  artificially increase the length of the cylinder, as in Remark \ref{Rem}. This does not affect any of the above arguments.
 
Next, set $B=\Del_\om\setminus(A_{\om,1}\cup\cdots\cup A_{\om,j_\om})$. Then
\[
m_\om(B)\geq m_\om\left(\Del_\om\setminus(\cup_{i=1}^{k_\om}Q_{\om,i})\right)\geq\del'.
\]
Since $h_\om$ is uniformly bounded away from $0$, we can find a lower bound $\del$ as desired (which depends on $\ve'$ through $\del'$). Now we will bound the $\tilde m_\om$-measure of $B$ from above. For any integer $K>1$ we have
\begin{eqnarray*}
\tilde m_\om(B)=m_\om (v\bbI_B)\leq \tilde m_\om(\cup_{\ell\geq K}\Del_{\om,\ell})+v_Km_\om(B).
\end{eqnarray*}
Now, let $c>0$ be so that $h_\om\geq c^{-1}$. Then with $Q_\om=Q_{\om,1}\cup Q_{\om,2}\cdots\cup Q_{\om,k_\om}$,  
\[
m_\om(B)\leq c\mu_\om(B)\leq c\sum_{j=0}^{s-1}\mu_\om\left((F_\om^j)^{-1}(Q_{\sig^j\om})\right)
= c\sum_{i=0}^{s-1}\mu_{\sig^i\om}(\Del_{\sig^i\om}\setminus Q_{\sig^i\om})\leq cv_0^{-1}s\ve'.
\]
In the last inequality we have used (\ref{ve '}) with $\ve'$ instead of $\ve$, and that $m_\om=v^{-1}d\tilde m_\om\leq v_0^{-1}\tilde m_0$.
Therefore,
\[
\tilde m_\om(B)\leq  \tilde m_\om(\Del_\om\cup_{\ell\geq K}\Del_{\om,\ell})+v_Kv_0^{-1}cs\ve'.
\]
In order to complete the proof 
we first take $K$ so that $ \tilde m_\om(\cup_{\ell\geq K}\Del_{\om,\ell})<\ve/2$ for a.e. $\om$, and then take $\ve$'s small enough so that $v_Kcs\ve'<v_0\ve/2$.
\end{proof}

We will also need the following
\begin{lemma}\label{Lmix}
Suppose that $\lim_{k\to\infty}d_k=0$. 
Assume also that  (\ref{UnfT}) holds true and that Assumption \ref{AsCov} holds true. 
For any $\ve$ and $s$, let $A_{\om,i}, 1\leq i\leq j_\om\leq M$ be the sets from Proposition \ref{PrpCn1} set $A_{\om,j_\om+1}$ to be the complement of their union. Let $\rho>0$. Then there exists $k_0>s$  which depends only on $\ve,s$ and $\rho$ so that for all $k\geq k_0$, $1\leq i\leq j_\om+1$ and $1\leq u\leq j_{\sig^k\om}+1$ we have
\begin{equation}\label{k geq k0}
\left|\frac{\tilde m_\om\left(A_{\om,i}\cap (F_{\om}^k)^{-1}A_{\sig^k\om,u}\right)}{\tilde m_\om(A_{\om,i}) \mu_{\sig^k\om}(A_{\sig^k\om,u})}-1\right|\leq\rho.
\end{equation}
\end{lemma} 
\begin{proof}
Since the denominator in the above fraction is bounded from below by some $\del$ which depends only on $\ve$ and $s$ (using that $\tilde m_\om \geq v_0 m_\om$), it is enough to show that the difference between the numerator and the denominator converges to $0$ when $k\to\infty$ uniformly in $\om$, $i$ and $u$. Fix some $k>s$ and some $i$ and $u$ as above. Next, for any $\ell>0$ we have
\[
\tilde m_\om\left(A_{\om,i}\cap (F_{\om}^k)^{-1}A_{\sig^k\om,u}\right)=m_\om\left(v^{(\ell)}\bbI_{A_{\om,i}}\bbI_{A_{\sig^k\om,u}}\circ F_\om^k\right)+O(\del_\ell)
\]
where $\del_\ell=\text{ess-sup}_\om \tilde m_\om(\cup_{j\geq\ell}\Del_{\om,\ell})$ which converges to $0$ as $\ell\to\infty$ and $v^{(\ell)}=v\bbI_{\cup_{j\leq\ell}\Del_{\om,j}}$. Moreover, 
\[
m_\om\left(v^{(\ell)}\bbI_{A_{\om,i}}\bbI_{A_{\sig^k\om,u}}\circ F_\om^k\right)=
m_{\sig^k\om}\left(P_{\sig^s\om}^{0,k-s}(\zeta)\bbI_{A_{\sig^k\om,u}}\right)
\]
where
\[
\zeta=P_\om^{0,s}(v^{(\ell)}\bbI_{A_{\om,i}}).
\]
Using Corollary \ref{CorLY} we have
\[
\|\zeta\|_{Li}\leq C(v_\ell)^2.
\]
Therefore,
\[
m_{\sig^k\om}(|P_{\sig^s\om}^{0,k-s}(\zeta)-m_{\sig^s\om}(\zeta)h_{\sig^k\om}|)\leq C(v_\ell)^2 d_{k-s}.
\]
Notice that 
\[
m_{\sig^s\om}(\zeta)=m_{\om}(v\bbI_{A_{\om,i}})-m_{\om}((v-v^{(\ell)})\bbI_{A_{\om,i}})=
\tilde m_\om(A_{\om,i})+O(\del_\ell).
\]
We conclude that
\[
\left|\tilde m_\om\left(A_{\om,i}\cap (F_{\om}^k)^{-1}A_{\sig^k\om,u}\right)-\tilde m_\om(A_{\om,i}) \mu_{\sig^k\om}(A_{\sig^k\om,u})\right|\leq O(\del_\ell)+C(v_\ell)^2 d_{k-s}.
\]
The proof of the lemma is completed by taking $\ell$ so that $\del_\ell<\rho/2$ and then $k_0>s$ so that
$C(v_\ell)^2 d_{k-s}<\rho/2$ for all $k>k_0$. 
\end{proof}

\subsection{Equvariant complex cones on random towers and the RPF theorem}\label{C-sec}
In this section we will work under Assumption \ref{AsCov}.
Moreover,, we will focus on the exponential case, and assume that there are $c_1,c_2>0$ so that $\bbP$-a.s. for all $n\geq1$ we have
\begin{equation}\label{Exp}
m_\om(R_\om\geq n)\leq c_1e^{-c_2n}.
\end{equation}
In particular by Theorem \ref{AssMixing} the sequence $d_k$ decays exponentially fast to $0$.
In this case we take $v_\ell=e^{\ve_0\ell}$ where $\ve_0<c_2$. Then, it is clear that (\ref{A1}) and (\ref{UnfT}) hold true.

Define the ``weighted" transfer operators $\cL_\om^{z}$, $z\in\bbC$ by $\cL_\om^{z}g=P_\om^{z}(gv)/v$ and for any $n$  set
\[
\cL_\om^{z,n}=\cL_{\sig^{n-1}\om}^{z}\circ\cdots\circ\cL_{\sig\om}^{z}\circ\cL_\om^{z}
\] 
which satisfy $\cL_{\om}^{z,n}g=P_\om^{z,n}(gv)/v$. Then Proposition \ref{Prop LY} means that the operators $\cL_\om^{it,n}$ are continuous with respect to the   norm $\|\cdot\|_{Li}$ (indeed $\|gv\|_\om=\|g\|_{Li,\om}$). Note that $\cL_\om=\cL_{\om}^{0}$ is the dual operators of $F_\om$ with respect to the measures $\tilde m_\om$ and $\tilde m_{\sig\om}$, that is for any bounded function $f$ and integrable function $g$,
\begin{equation}\label{L Du}
\int f\cL_\om g d\tilde m_{\sig\om}=\int g\cdot f\circ F_\om d\tilde m_\om.
\end{equation}

Note also that with $\tilde h_\om=h_\om/v$ we have $\mu_\om=\tilde h_\om d\tilde m_\om$, where $h_\om$ is the random density function of the equivariant measures $\mu_\om$ from Proposition \ref{MixProp}.

Our main goal in this section is to prove the following theorem.
\begin{theorem}\label{RPF}
Suppose that (\ref{Exp}) holds true and that Assumption \ref{AsCov} holds.
There exists a constant $r>0$, which depends only on the initial parameters, so that for every $z\in B(0,r):=\{\zeta\in\bbZ:\,|\zeta|<r\}$
there exist random measurable triplets depending only on the operators $\cL_\om^{z}$
consisting of a nonzero complex number 
$\la_\om(z)$, a complex function $h_\om^{(z)}\in\cH_\om$ and a 
complex continuous linear functional $\nu_\om^{(z)}\in\cH_\om^*$ such that:

(i) For $\bbP$-a.e. $\om$,
 $\la_\om(0)=1$, $h_\om^{(0)}=\tilde h_\om$, $\nu_j^{(0)}=\tilde m_\om$
 and for any $z\in B(0,r)$,
\begin{equation}\label{RPF deter equations-General}
\cL_\om^{z} h_\om^{(z)}=\la_\om(z)h_{\sig\om}^{(z)},\,\,
(\cL_\om^{z})^*\nu_{\sig\om}^{(z)}=\la_\om(z)\nu_{\om}^{(z)}\text{ and }\,
\nu_\om^{(z)}(h_\om^{(z)})=\nu_\om^{(z)}(h_\om^{(0)})=1.
\end{equation} 
When $z=t\in\bbR$ and $|t|<r$ then 
 $\la_\om(t)>a$ for some constant $a$ not depending on $\om$ and $t$. 
 Moreover, $\nu_\om^{(t)}$ is a positive measure (which assigns positive mass to open subsets of $\Del_\om$) and the equality 
$\nu_{\sig\om}^{(t)}\big(\cL_\om^{t} g)=\la_\om(t)\nu_{\om}^{(t)}(g)$ holds true for any 
bounded Borel function $g:\Del_\om\to\bbC$.

(ii) Set $U=B(0,r)$. Then the maps 
\[
\la_\om(\cdot):U\to\bbC,\,\, h_\om^{(\cdot)}:U\to \cH_{\om}\,\,\text{ and }
\nu_\om^{(\cdot)}:U\to \cH_{\om,}^*
\]
are analytic, where $\cH_{\om,}^*$ is the dual of $\cH_{\om}$.  Moreover, there exists a constant $C>0$, which depends only on the initial parameters
such that 
\begin{equation}\label{UnifBound}
\max\big(\sup_{z\in U}|\la_\om(z)|,\, 
\sup_{z\in U}\| h_\om^{(z)}\|_{Li},\, \sup_{z\in U}
\|\nu^{(z)}_\om\|\big)\leq C,
\end{equation}
where $\|\nu\|$ is the 
operator norm of a linear functional $\nu:\cH_\om\to\bbC$.

(iii) There exist constants $A>0$ and $\del\in(0,1)$, 
which depend only on the initial parameters,
 so that $\bbP$-a.s. for any $g\in\cH_\om$
and $n\geq1$,
\begin{equation}\label{Exponential convergence}
\Big\|\frac{\cL_\om^{z,n}g}{\la_{\om,n}(z)}-\nu_\om^{(z)}(g) h^{(z)}_{\sig^n\om}\Big\|_{Li}\leq A\|g\|_{Li}\del^n
\end{equation}
where $\la_{\om,n}(z)=\la_{\om}(z)\cdot\la_{\sig\om}(z)\cdots\la_{\sig^{n-1}}(z)$. 
\end{theorem}
Note that for any two functions $g:\Del_\om\to\bbR$ and $f:\Del_{\sig^n\om}\to\bbR$ we have
\begin{eqnarray*}
\mu_\om(g\cdot f\circ F_\om^n)=\tilde m_{\sig^n\om}\left(f\cdot \cL_{0}^{\om,n}(g\tilde h_\om)\right)\\=
\mu_\om(g)\mu_{\sig^n\om}(f)+\tilde m_{\sig^n\om}\left(f\left(\cdot \cL_{0}^{\om,n}(g\tilde h_\om-\tilde m_\om(g\tilde h_\om)\tilde h_{\sig^n\om}\right)\right).
\end{eqnarray*}
Therefore, using (\ref{Exponential convergence}) together with $\|\tilde h_\om g\|_{Li}\leq 3\|g\|_{Li}\|\tilde h_\om\|_{Li}\leq C\|g\|_{Li}$, we get that there is a constant $A_0>0$ so that
\begin{equation}\label{ExpCor}
\left|\mu_\om(g\cdot f\circ F_\om^n)-\mu_\om(g)\mu_{\sig^n\om}(f)\right|\leq A_0\|g\|_{Li}\|f\|_{L^1(\mu_{\sig^n\om})}\del^n.
\end{equation}



\subsection{Proof of Theorem \ref{RPF}}

For every $\ve>0$ and $s\geq1$ we consider the partitions $A_{\om,i}$ of $\Del_\om$ from Proposition \ref{PrpCn1}, where $1\leq i\leq j_\om+1$. Let us denote this partition by $\cP_\om(\ve,s)$.
For any $a,b,c>1$ 
let $\cC_{\om,a,b,c}=\cC_{\om,a,b,c,\ve,s}$ be the real cone consisting of all functions $g:\Del_\om\to\bbR$ so that 
\vskip0.2cm
\begin{itemize}
\item
$0\leq\frac{1}{\mu_\om(P)}\int_P gd\tilde m_{\om}\leq a\int gd\tilde m_\om;\,\,\forall\,P\in\cP_{\om}(\ve,s)$.
\\
\item
$|g|_\om=|g|_{\om,\be}\leq b\int gd\tilde m_\om$.
\\
\item
$|g(x)|\leq c\int gd\tilde m_{\om},\,\,\text{for any }\,x\in A_{\om,j_\om+1}$.
\end{itemize}

As in \cite{M-D} we have the following result.
\begin{proposition}
For any $a,b,c>1$, $\ve>0$ $s\in\bbN$  and $\del\in(0,1)$
the real projective diameter of $\cC_{\om,\del a,\del b,\del c,\ve,s}$ inside $\cC_{\om,a,b,c,\ve,s}$ does not exceed a constant $r=r(a,b,c,\del,\ve,s)$ which depends only on $a,b,c,s,\ve$ and $\del$. 
\end{proposition}

The next step in the proof of Theorem \ref{RPF} is the following result.

\begin{proposition}\label{ConeProp1}
Suppose that (\ref{Exp}) holds true and that Assumptions \ref{MixProp} and \ref{AsCov} are satisfied.
Then there are $\ve>0$, $s,k_1\in\bbN$, $a,b,c>1$ and $\del\in(0,1)$ so that for $\bbP$-a.a. $\om$ and $k\geq k_1$ we have
\begin{equation}\label{ConInv}
\cL_\om^{0,k}\cC_{\om,a,b,c,\ve,s}\subset\cC_{\sig^k\om,\del a,\del b,\del c,\ve,s}.
\end{equation}
\end{proposition}
In fact if $\ve$ is small enough and $s,k$, $a$ $b/a$ and $c/a$ are large enough we can find $k_1$ so that (\ref{ConInv}) holds true for $\bbP$-a.a. $\om$ and $k\geq k_1$ with $\del=1/2$.

\begin{proof}
Let $\ve>0$, $s,k\in\bbN$, $a,b,c>1$ and $g\in\cC_{\om,a,b,c,\ve,s}$.
In order to show that $\cL_\om^k g=\cL_\om^{0,k}g$ satisfies the first desired condition, for any $P=A_{\sig^k\om,q}\in\cP_{\sig^k\om}$, $1\leq q\leq j_{\sig^k\om}+1$ we first write
\begin{eqnarray*}
\frac {1}{\mu_{\sig^k\om}(P)}\int_P\cL_\om^k gd\tilde m_{\sig^k\om}=\frac {1}{\mu_{\sig^k\om}(P)}\int_{(F_\om^k)^{-1}P}gd\tilde m_{\om}\\=\sum_{i=1}^{j_\om}\frac {1}{\mu_{\sig^k\om}(P)}\int_{A_{\om,i}\cap(F_\om^k)^{-1}P}gd\tilde m_{\om}+\frac {1}{\mu_{\sig^k\om}(P)}\int_{A_{\om,j_\om+1}\cap(F_\om^k)^{-1}P}gd\tilde m_{\om}.
\end{eqnarray*}
Next, let $\rho\in(0,1)$. Given $\ve,s$ and $\rho$ by Lemma \ref{Lmix} there is $k_0=k_0(\ve,s,\rho)$ so that (\ref{k geq k0}) holds true for any $k>k_0$. Using that $g\in\cC_{\om,a,b,c,\ve,s}$ and some standard estimates we obtain exactly as in the proof of \cite[Proposition 3.7]{M-D} that for all $1\leq i\leq j_\om$,
\[
\frac {1}{\mu_{\sig^k\om}(P)}\int_{A_{\om,i}\cap(F_\om^k)^{-1}P}gd\tilde m_{\om}\leq (1+\rho)\left(\int_{A_{\om,i}} gd\tilde m_\om+b\beta^s\tilde m_\om(A_{\om,i})\int gd\tilde m_\om\right)
\]
and 
\[
(1-\rho)\Big(\int_{A_{\om,i}} gd\tilde m_\om-(1+\rho)b\beta^s\tilde m_\om(A_{\om,i})\Big)\int gd\tilde m_\om\big)\leq \frac {1}{\mu_{\sig^k\om}(P)}\int_{A_{\om,i}\cap(F_\om^k)^{-1}P}gd\tilde m_{\om}.
\]
Moreover,
\begin{eqnarray*}
(1-\rho)\int_{A_{\om,j_\om+1}}gd\tilde m_\om-2c(1+\rho)\ve\int gd\tilde m_\om\leq\frac {1}{\mu_{\sig^k\om}(P)}\int_{A_{\om,j_\om+1}\cap(F_\om^k)^{-1}P}gd\tilde m_{\om}\\\leq (1+\rho)c\ve\int gd\tilde m_\om.
\end{eqnarray*}
Observe that 
\[
\int_P\cL_\om^k gd\tilde m_{\sig^k\om}=\int_{(F_\om^k)^{-1}P}gd\tilde m_\om.
\]
Therefore, by spiting the above integral  according to the partition $A_{\om,i}$ and
summing these inequalities we get
\begin{eqnarray*}
(1-\rho)\left(1-c\ve-(1+\rho)\be^s b-2(1+\rho)c\ve\right)\int gd\tilde m_\om\leq \frac {1}{\mu_{\sig^k\om}(P)}\int_P\cL_\om^k gd\tilde m_{\sig^k\om}\\\leq \left(1+\rho)(1+\be^sb+c\ve\right)\int gd\tilde m_\om
\end{eqnarray*}
Since 
\[
\int gd\tilde m_\om=\int  \cL_\om^k gd\tilde m_{\sigma^k\om}
\]
for any given $\del,a,b$ and $c$ so that $\del a>1$,
we get that the function $\cL_\om^k g$ would satisfy the first condition in the definition of the cone $\cC_{\sig^k\om,\del a,\del b,\del c,\ve,s}$ if $\ve,\be^s$ and $\rho$ are small enough and $k>k_0(\ve,s,\rho)$ (so far when $\del=1/2$ our only restriction is that $a,b,c$ are large enough).

Now we will verify the second condition. Let $x=(x,\ell),y=(y,\ell)\in\Del_{\sig^k\om}$. If $k\leq \ell$ then 
\begin{eqnarray*}
|\cL_\om^kg(x,\ell)-\cL_\om^kg(y,\ell)|=v_{\ell-k}|g(x,\ell-k)-g(y,\ell-k)|/v_\ell\\=e^{-\ve_0 k}|g(x,\ell-k)-g(y,\ell-k)|\leq e^{-\ve_0 k}\be^{k} |g|_\be d_{\sig^k\om}(x,y).
\end{eqnarray*}
If $k>\ell$ then with $g_v=vg$ by (\ref{LY2.2}) we have
\begin{eqnarray*}
|\cL_\om^kg(x,\ell)-\cL_\om^kg(y,\ell)|=e^{-\ve_0\ell}|P_\om^{0,k}g_v(x,\ell)-P_\om^{0,k}g_v(y,\ell)|\\\leq e^{-\ve_0\ell}Q(C_1+2\be^{-1})(\|g\|_{L^1(\tilde m_\om)}+C_2\be^k|g|_\be)d_{\sig^k\om}(x,y)
\end{eqnarray*}
where we have used that $\|gv\|_s=\|g\|_\infty$, $\|gv\|_h=|g|_\om$ and 
\[
\int |gv|d m_\om=\int |g|d\tilde m_\om.
\]
Observe that 
\[
\int_{A_{\om,j_\om+1}}|g|d\tilde m_\om\leq \|g\bbI_{A_{\om,j_\om+1}}\|_\infty\tilde m_\om(A_{\om,j_\om+1})\leq \ve c\int gd\tilde m_\om.
\]
Moreover, for any $1\leq i\leq j_\om$ and $x\in A_{\om,i}$ we have 
\begin{equation}\label{Est Q}
\left|g(x)-\frac1{\tilde m_\om(A_{\om,i})}\int_{A_{\om,i}}gd\tilde m_\om\right|\leq |g|_\om\be^s\leq b\be^s\int gd\tilde m_\om
\end{equation}
since the diameter of $Q_{\om,i}$ does not exceed $\be^s$. Notice that 
\[
\frac1{\tilde m_\om(A_{\om,i})}\leq \frac {D_0}{ \mu_\om(A_{\om,i})}
\]
for some constant $D_0$. Indeed,
\[
\mu_\om(A_{\om,i})=m_\om(\bbI_{A_{\om,i}}/h_\om)\leq cm_\om(A_{\om,i})\leq c\tilde m_\om(A_{\om,i})
\]
where $c>0$ satisfies $h_\om\geq c^{-1}>0$. Therefore, 
\[
\|g\bbI_{A_{\om,i}}\|_\infty\leq D_0\frac1{\mu_\om(A_{\om,i})}\int_{A_{\om,i}}gd\tilde m_\om+ b\be^s\int gd\tilde m_\om\leq (D_0a+ b\be^s)\int gd\tilde m_\om.
\]
Hence, 
\begin{eqnarray}\label{Intest}
\int|g|d\tilde m_\om=\sum_{i=1}^{j_\om}\int_{A_{\om,i}}|g|d\tilde m_\om+\int_{A_{\om,j_\om+1}}|g|d\tilde m_\om\\\leq\sum_{i=1}^{j_\om}\tilde m_\om(A_{\om,i})(D_0a+ b\be^s)\int gd\tilde m_\om+
\ve c \tilde m_\om(A_{\om,j_om+1})gd\tilde m_\om\nonumber\\
\leq c_0(\ve c+b\be^s+D_0a)\int gd\tilde m_\omega\nonumber
\end{eqnarray}
where $c_0=\text{ess-sup }\tilde m_\om(\Del_\om)<\infty$.
We conclude that when $k>\ell$ then
\[
|\cL_\om^kg(x,\ell)-\cL_\om^kg(y,\ell)|\leq C(D_0a+ b\be^s+b\be^k+c\ve)\int gd\tilde m_\om \,\cdot d_{\sig^k\om}(x,y)
\]
for some $C>0$ which does not depend on $\om,\ve,s,k,\rho,a,b$ and $c$. If we take $a$ and $b$ so that
$CD_0a<b/4$ and then $\ve$ small enough and $k$ and $s$ large enough so that $b/4+Cb(\be^s+\be^k)+c\ve<b/4$
then the constant on the above right hand side does not exceed $b/2$. 

So far we have shown $\cL_\om^kg$ satisfies the first two conditions defining $\cC_{\sig^k\om,\del a,\del b, \del c,\ve,s}$ with $\del=1/2$ if $k$ and $s$ are large enough, $\ve$ is small enough (uniformly in $\om$) and $CD_0a<b/4$. Now we will show that for many choices of parameters the third condition also holds true. Let $(x,\ell)\in A_{\sig^k\om, j_{\sig^k\om}+1}$. If $k>\ell$ then 
\[
|\cL_\om^k g(x,\ell)|=e^{-\ve_0k}|g(x,k-\ell)|.
\] 
The above arguments show that, in fact $|g|\leq E\int gd\tilde m_\om$  for some constant $E>0$ (the values of $|g|$ on $QA_{\om,i}$ for $1\leq i\leq j_\om$ are estimated using (\ref{Est Q}) and what proceeds it). Therefore,
\[
|\cL_\om^k g(x,\ell)|\leq Ee^{-\ve_0 k}\int gd\tilde m_\om<\frac12 a\int gd\tilde m_\om
\]
if $k$ is large enough. Assume now that $k\leq \ell$. Then
\[
|\cL_\om^k g(x,\ell)|=e^{-\ell v_0}|P_\om^{0,k}g_v(x,\ell)|.
\]
Using (\ref{LY2.1}) we have 
\[
|P_\om^{0,k}g_v(x,\ell)|\leq Q\left(\int |g|d\tilde m_\om+\be^kC_2|g|_\om\right).
\]
Using (\ref{Intest}), we see that if also $aCQD_0<c/4$, $\ve$ is small enough and $k$ and $s$ are large enough then 
\[
\sup_{x\in A_{\sig^k\om,j_{\sig^k\om}+1}}|\cL_\om^k g(x)|\leq \frac12 c\int gd\tilde m_\om=
 \frac12 c\int \cL_\om^k gd\tilde m_{\sig^k\om}.
\]
and we conclude that the proposition holds true with $\del=1/2$ for a.e. $\om$, whenever $\ve$ is small enough and $s,k$, $b/a$ and $c/a$ are large enough.
\end{proof}

Let $a,b,c,\ve,s,k_1$ and $\del$ satisfy (\ref{Const M}) for any $k\geq k_1$. Set $\cC_\om=\cC_{\om,a,b,c,s,\ve}$, and denote by $\cC_{\om,\bbC}$ the canonical complexification\footnote{We refer to \cite{Rug} for the definition of a canonical complexification. See also \cite[Appendix A]{HK} for a summary of all the properties of real and complex cones which will be used in what follows.} of the real cone $\cC_\om$. 
The proof of Theorem \ref{RPF} is completed by applying the following theorem together with \cite[Theorem 4.1]{HK} and \cite[Theorem 4.2]{HK}.

\begin{theorem}\label{ConeThm}
Suppose that (\ref{Exp}) and  hold true. Then,
if $a, b/a$ and $c/a$ are large enough then the following holds true:

(i) The cone $\cC_{\om,\bbC}$ is linearly convex, and it contains the functions $\tilde h_\om=h_\om/v$ and $\textbf{1}$ (the function which takes the constant value $1$). Moreover,  the measure $\tilde m_\om$, when viewed as a linear functional, is a member of the dual complex cone $\cC_{\om,\bbC}^*$ and
the cones  $\cC_{\om,\bbC}$ and $\cC_{\om,\bbC}^*$ have bounded aperture. In fact,
there exist constants $K,M>0$ so that for any $f\in\cC_{\om,\bbC}$ and $\mu\in\cC_{\om,\bbC}^*$, 
\begin{equation}\label{aperture0}
\|f\|\leq K|\tilde m_\om(f)|
\end{equation}
and
\begin{equation}\label{aperture}
\|\mu\|\leq M|\mu(\tilde h_\om)|.
\end{equation}
Here $\|f\|=\|f\|_{Li}$ and $\|\mu\|$ is the corresponding operator norm (all of the above hold true $\bbP$-a.s. and the constant do not depend on $\om$).

(ii) The cone $\cC_{\om,\bbC}$ is reproducing. In fact, there exists a constant $K_1$ so that $\bbP$-a.s. for every $f\in\cH_\om$ bounded there exists $R(f)\in\bbC$ such that $|R(f)|\leq K_1\|f\|$
and 
\[
f+R(f)\tilde h_\om \in\cC_{\om,\bbC}.
\]

(iii) There exist  constants $r>0$ and $d_1>0$ so that $\bbP$-a.s. for every  complex number $z$ with $|z|<r$ and $k_1\leq k\leq 2k_1$, where $k_1$ comes from Proposition \ref{ConeProp1}, we have 
\[
\cL_\om^{z,k}\cC_{\om,\bbC}'\subset\cC_{\sig^k\om,\bbC}'
\]
and 
\[
\sup_{f,g\in\cC'_{\om,\bbC}}\del_{\cC_{\sig^k\om,\bbC}}(\cL_\om^{z,k}f,\cL_\om^{z,k}g)\leq d_1
\]
where $\cC'=\cC\setminus\{0\}$ for any set of functions, and $\del_{\cC_{\sig^k\om,\bbC}}$ is the complex projective metric corresponding to the complex cone  $\cC_{\sig^k\om,\bbC}$ (see  \cite[Appendix A]{HK}).
\end{theorem}

\begin{proof}
The proof proceeds similarly to the proof of  \cite[Theorem 6.3]{LLT2020}. For readers' convenience we will give most of the details. 
We begin with the proof of the first part. First, since 
\[
\int_A \tilde h_\om d\tilde m_\om=\int_Ad\mu_\om=\mu_\om(A),
\]
for any measurable  set $A$, 
it is clear that $\tilde h_\om\in\mathcal C_\om$ if $a>1$, $b>|\tilde h_\om|_\om$ and $c>\|\tilde h_\om\|_\infty$ (note that $|\tilde h_\om|_\om$ and $\|\tilde h_\om\|_\infty$ are uniformly bounded in $\om$). Moreover, if $c>1$ and $a>D$, where
\begin{equation}\label{D def}
D=\text{ess-sup}\max\Big\{\frac{\tilde m_\om(P)}{\mu_\om(P)}:\,P\in\cP_\om\Big\}<\infty
\end{equation}
then $\textbf{1}\in\mathcal C_{\om}$ (the above essential supremum is indeed finite since $\mu_\om(A_{\om,i})\geq\del(\ve,s)>0$ by (\ref{MIN})).

Next, if $f\in\cC_\om'$ and $\tilde m_\om(f)=0$ then by (\ref{f bound}) we have $f=0$ and so $\tilde m_\om \in\cC_\om^*$. In fact, we have that
\begin{equation}\label{f bound}
\|f\|_\infty\leq c_2\int fd\tilde m_\om
\end{equation}
for some $c_2>0$, and so it follows
from the definitions of the norm $\|f\|_{Li}$ and from (\ref{f bound}) that 
\[
\|f\|=\|f\|_\infty+\sup_\ell|f|_{\om,\Del_{\om,\ell}}=\|f\|_\infty+|f|_\om\leq (c_2+b)\tilde m_\om(f)=(c_2+b)\int fd\tilde m_\om.
\]
and therefore by  \cite[Lemma 5.3]{Rug} the inequality (\ref{aperture0}) hold true with $K=2\sqrt 2(c_2+b)$. According to Lemma A.2.7 \cite[Appendix A]{HK}, for any $M>0$,
inequality (\ref{aperture}) holds true for every $\mu\in\cC_{\om,\bbC}^*$ if
\begin{equation}\label{Const M}
B_{\om,\cH}(\tilde h_\om,1/M):=\left\{f\in\cH_\om: \|f-\tilde h_\om\|_{Li,\om}<\frac1M\right\}\subset\cC_{\om,\bbC}.
\end{equation}
Now we will find  a constant $M$ for satisfying (\ref{Const M}). Fix some $\om\in\Om$.
For any $f$ with $\|f\|_{Li}<\infty$, $P\in\cP_\om$ and $x_1\in A_{\om,j_\om+1}$, and distinct $x,y$ which belong to the same level  $\Del_{\om,\ell}$ (for some $\ell$) set
\begin{eqnarray*}
\Upsilon_P(f)=\frac1{\mu_\om(P)}\int_P fd\tilde m_\om,\,\,
\Gam_P(f)=a\int fd\tilde m_\om-\frac1{\mu_\om(P)}\int_P fd\tilde m_\om,\\
\Gam_{x,y}(f)=b\int fd\tilde m_\om-\frac{f(x)-f(y)}{d_\om(x,y)}
\, \text{ and }\,
\Gam_{x_1,\pm}(f)=c\int fd\tilde m_\om\pm f(x_1)
\end{eqnarray*}
Let $\Gam_\om$ be the collection of all the above linear functionals. Then, with $\cH_{\om}(\bbR)=\cH_{\om,\be}(\bbR)$ denoting the space of real valued $f:\Del_\om\to\bbC$ with $\|f\|_{Li}=\|f\|_{Li,\om}<\infty$,
\[
\cC_\om=\{f\in\cH_{\om}(\bbR):\,\gamma(f)\geq0,\,\forall \gam\in\Gam_\om\}
\]
and so
\begin{equation}\label{Complexification1}
\cC_{\om,\bbC}=\{f\in \cH_\om\:\,\Re\big(\overline{\mu(f)}\nu(f)\big)
\geq0\,\,\,\,\forall\mu,\nu\in\Gam_\om\}.
\end{equation}
where as defined earlier $\cH_\om=\cH_\om(\bbC)$ is the corresponding space of complex functions.
Let $g\in\cH_\om$ be of the form $g=\tilde h_\om+q$ for some $q\in\cH_\om$. We need to find a constant $M>0$ so that $\tilde h_\om+q\in \cC_{\om,\bbC}$ if $\|q\|<\frac1M$. In view of (\ref{Complexification1}), there are several cases to consider. First, suppose that $\nu=\Upsilon_{P}$ and $\mu=\Upsilon_Q$ for some $P,Q\in\cP_\om$. Since
\[
\frac{1}{\mu_\om(A)}\int_A \tilde h_\om d\tilde m_\om=\frac{1}{\mu_\om(A)}\int_A 1d\mu_\om=1
\]
for any measurable set $A$ with positive measure, we have 
\[
\Re\big(\overline{\mu(\tilde h_\om+q)}\nu(\tilde h_\om+q)\big)\geq 1-(D^2\|q\|^2+2D\|q\|)
\]
where $D$ was defined in \ref{D def} and $\|\cdot\|=\|\cdot\|_{Li}$.
Hence
\[
\Re\big(\overline{\mu(\tilde h_\om+q)}\nu(\tilde h_\om+q)\big)>0,
\]
if $\|q\|$ is sufficiently small. Now consider the case when $\mu=\Upsilon_P$ for some $P\in\cP_\om$ and
$\nu$ is one of the $\Gamma$'s, say $\nu=\Gam_{x,y}$. Then
\begin{eqnarray*}
\Re\big(\overline{\mu(\tilde h_\om+q)}\nu(\tilde h_\om+q)\big)\geq b-\|\tilde h_\om\|-bc_0\|q\|-\|q\|\\-D\|q\|(b+\|\tilde h_\om\|+bc_0\|q\|+\|q\|)
\geq b-\|\tilde h_\om\|-C(D,b)(\|\tilde h_\om\|+\|q\|+\|q\|)^2 
\end{eqnarray*}
where $C(D,b,c_0)>0$ depends only on $D$, $b$ and $c_0:=\text{ess-sup }\tilde m_\om(\textbf{1})<\infty$. If $\|q\|$ is sufficiently small and $b>\|\tilde h_\om\|$ then 
the above left hand side is clearly positive. Similarly, if $\text{ess-sup }\|\tilde h_\om\|<\frac12\min\{a,b,c\}$ and $\|q\|$ is sufficiently small then 
\[
\Re\big(\overline{\mu(\tilde h_\om+q)}\nu(\tilde h_\om+q)\big)>0
\]
when either $\nu=\Gam_{x_1,\pm}$ or $\nu=\Gam_{x,y}$ (note that $\om\to\|\tilde h_\om\|$ is a bounded random variable).

Next, consider the case when $\mu=\Gam_{x_1,\pm}$ for some $x_1\in A_{\om,j_\om+1}$ and  $\nu=\Gam_{x,y}$ for some distinct $x$ and $y$ in the same floor. Then with some constant $A>0$ which depends only on $c,b$ and $c_0$ we have
 \begin{eqnarray*}
\Re\big(\overline{\mu(\tilde h_\om+q)}\nu(\tilde h_\om+q)\big)\geq 
bc-\|\tilde h_\om\|^2-A\|q\|
\end{eqnarray*}
where we have used that $\int \tilde h_\om d\tilde m_\om=1$ and that $\|\tilde h_\om\|$ is bounded. Therefore, if $\|q\|$ is sufficiently small and $c$ and $b$ are sufficiently large then
\[
\Re\big(\overline{\mu(\tilde h_\om+q)}\nu(\tilde h_\om+q)\big)>0.
\] 
Similarly, since 
\[
\left|\frac1{\mu_\om (P)}\int_P qd\tilde m_\om\right|\leq D\|q\|
\]
and 
\[
\int qd\tilde m_\om\leq \tilde m_\om(\textbf{1})\|q\|\leq c_0\|q\|,
\]
when  $a,b$ and $c$ are large enough 
there are constants $A_1,A_2>0$ which depend only on $a,b,c,D$, $c_0$ and $\text{ess-sup }\|\tilde h_\om\|$ so that
for any other choice of $\mu,\nu\in\Gamma_\om\setminus\{\Upsilon_P\}$ and $q$ with $\|q\|\leq1$
we have
\begin{equation*}
\Re\big(\overline{\mu(\tilde h_\om+q)}\nu(\tilde h_\om+q)\big)\geq A_1(1-A_2\|q\|)
\end{equation*}
and so, when  $\|q\|$ is sufficiently small then the above left hand side is positive. The proof of Theorem \ref{ConeThm} (i) is now complete. 

The proof of Theorem \ref{ConeThm} (ii) proceeds exactly as the proof of \cite[Lemma 3.11]{M-D}: for a real valued function $f\in\cH$, we have that $f+R(f)\tilde h_\om$ for $R(f)>0$ belongs to the cone  if 
\begin{eqnarray*}
R(f)\geq(a-1)^{-1}\cdot\max\Big\{\frac1{\mu_\om(P)}\int_P fd\tilde m_\om-a\int fd\tilde m_\om:\,\,P\in\cP_\om\Big\},\\
R(f)\geq\frac{|f|_\om-b\int fd\tilde m_\om}{b-|\tilde h_\om|_\om},\,\,R(f)>\max\Big\{-\frac1{\mu_\om(P)}\int_P fd\tilde m_\om:\,\,P\in\cP_\om\Big\}\,\,\text{ and }\\
R(f)\geq\frac{\|f\|_\infty-c\int fd\tilde m_\om}{c-\|\tilde h_\om\|_\infty}
\end{eqnarray*}
where we take $a,b$ and $c$ so that all the denominators appearing in the above inequalities  are bounded from below by, say $\frac12$,
and we have used that $\frac{1}{\mu_\om(A)}\int \tilde h_\om d\tilde m_\om=1$ for any measurable set $A$ (apply this with $A=P\in\cP_\om$). Now we will show that it is indeed possible to choose such $R(f)\leq K_1\|f\|$ for some constant $K_1$. We have
\[
\frac1{\mu_\om(P)}\int_P fd\tilde m_\om\leq D\|f\|_\infty\leq D\|f\|
\]
where $D$ is given by (\ref{D def}),
and 
\[
\int fd\tilde m_\om\leq \|f\|_\infty\tilde\mu_\om(\textbf{1})\leq \|f\|_\infty c_0\leq \|f\|c_0
\]
for some $c_0>0$. Therefore, when, say $a>2$ then all the above lower bounds on $R(f)$ are bounded from above by 
\[
2\max(D+ac_0, 1+bc_0,1+cc_0)\|f\|.
\]
Therefore, for real $f$'s we can take $K_1=2\max(D+ac_0, 1+bc_0,1+cc_0)$.
For complex-valued $f$'s we can write $f=f_1+if_2$, then take $R(f)=R(f_1)+iR(f_2)$ and use that with $\bbC'=\bbC\setminus\{0\}$, 
\[
\cC_{\om,\bbC}=\bbC'(\cC_\om+i\cC_\om).
\]

Now we will prove Theorem \ref{ConeThm} (iii). Let $k_1\leq k\leq 2k_1$, where $k_1$ comes from Proposition \ref{PrpCn1}.
According to Theorem A.2.4  in \cite[Appendix A]{HK} (which is \cite[Theorem 4.5]{Dub2}), if 
\begin{equation}\label{Comp1}
|\gam(\cL_\om^{z,k} f)-\gam(\cL_\om^{0,k}f)|\leq \ve_1 \gam(\cL_\om^{0,k}f)
\end{equation}
for any nonzero $f\in\cC_\om$ and $\gam\in\Gamma_{\sig^k\om}$, for some $\ve_1>0$ so that 
\[
\del:=2\ve_1\Big(1+\cosh\big(\frac12 d_0\big)\Big)<1
\]
where $d_0$ comes from Proposition \ref{PrpCn1},
then, with $\cC'_{\om,\bbC}=\cC_{\om,\bbC}\setminus\{0\}$,
\begin{equation}\label{I}
\cL_\om^{z,k}\cC'_{\om,\bbC}\subset\cC'_{\sig^k\om,\bbC}
\end{equation}
and
\begin{equation}\label{II} 
\sup_{f,g\in\cC_{\om,\bbC}}\del_{\sig^k\om}(\cL_\om^{z,k}f,\cL_\om^{z,k}g)\leq d_0+6|\ln(1-\del)|.
\end{equation}
We will show now that there exists a constant $r>0$ so that (\ref{Comp1}) holds true for any $z\in B(0,r)$ and $f\in\cC_\om$. This relies on the following very elementary result. 
\begin{lemma}\label{A A' lemma}
Let $A$ and $A'$ be complex numbers, $B$ and $B'$ be real numbers, and let $\ve_1>0$ and $\eta\in(0,1)$ so that
\begin{itemize}
\item
$B>0$ and $B>B'$;
\item
$|A-B|\leq\ve_1B$;
\item
$|A'-B'|\leq\ve_1 B $;
\item
$|B'/B|\leq\eta$.
\end{itemize}
Then 
\[
\left|\frac{A-A'}{B-B'}-1\right|\leq 2\ve_1(1-\eta)^{-1}.
\]
\end{lemma}
To prove Lemma \ref{A A' lemma}  we just write
\[
\left|\frac{A-A'}{B-B'}-1\right|\leq\left|\frac{A-B}{B-B'}\right|+
\left|\frac{A'-B'}{B-B'}\right|\leq \frac{2B\ve_1}{B-B'}=\frac{2\ve_1}{1-B'/B}.
\]
Next, let $f\in\cC_\om'$. First, suppose that $\gam$ have the form 
$\gam=\Gam_{P}$ for some $P\in\cP_{\sig^k\om}$. Set
\begin{eqnarray*}
A=a\int \cL_\om^{z,k}fd\tilde m_{\sig^k\om},\,\, A'=\frac1{\mu_{\sig^k\om}(P)}\int_P \cL_\om^{z,k}fd\tilde m_{\sig^k\om},\\
B=a\int\cL_\om^{0,k}fd\tilde m_{\sig^k\om}\,\,\text{ and }\,\,B'=\frac1{\mu_{\sig^k\om}(P)}\int_P \cL_\om^{0,k}fd\tilde m_{\sig^k\om}.
\end{eqnarray*}
Then $B=a\int f d\tilde m_\om$ (since $(\cL_\om^0)^*\tilde m_{\sig\om}=\tilde m_\om$) and
\[
|\gam(\cL_\om^{z,k}f)-\gam(\cL_\om^{0,k}f)|=|A-A'-(B-B')|.
\]
We want to show that the conditions of Lemma \ref{A A' lemma} hold true. 
By Proposition \ref{PrpCn1} we have 
\begin{equation}\label{Smaller cone}
\cL_\om^{0,k}f\in\cC_{\sig^k\om,\del a,\del b,\del c,s,\ve}
\end{equation}
which in particular implies that 
\[
0\leq B'\leq \del a\int\cL_\om^{0,k}fd\tilde m_{\sig^k\om}=\del B.
\]
Since $f$ is nonzero and $\int\cL_\om^{0,k}fd\tilde m_{\sig^k\om}=\int fd\tilde m_\om\geq 0$ the number $B$ is positive 
(since (\ref{aperture0}) holds true). It follows that $B>B'$ and that 
\[
|B'/B|\leq\del<1. 
\]
Now we will estimate $|A-B|$. For any complex $z$ so that $|z|\leq1$ write
\begin{eqnarray*}
|A-B|=a\left|\int\cL_\om^{0,k}\big(f(e^{zS_k^\om\varphi}-1)\big)d\tilde m_{\sig^k\om}\right|\leq a\|f\|_\infty\|e^{zS_k^\om\varphi}-1\|_\infty
\int\cL_\om^{0,k}\textbf{1}d\tilde m_{\sig^k\om}\\= a\|f\|_\infty\|e^{zS_k^\om\varphi}-1\|_\infty\int \textbf{1}d \tilde m_{\om}=
a \tilde m_\om(\textbf{1})\|f\|_\infty\|e^{zS_k^\om\varphi}-1\|_\infty\\
\leq C_2ac_2\int fd\tilde m_\om\,\cdot(2k_1e^{2k_1\|\varphi\|_\infty}\cdot|z|\|\varphi\|_\infty)\\=2ac_2k_1R\|\varphi\|_\infty|z|\int\cL_\om^{0,k}f d\tilde m_{\sig^k\om}=R_1|z|B
\end{eqnarray*}
where $\textbf{1}$ is the function which takes the constant value $1$, $C_2$ is an upper bound of $\tilde m_\om(\textbf{1})$, 
\[
\|\varphi\|_\infty:=\text{ess-sup}\|\varphi_\om\|_\infty
\] 
and 
\[
R_1=2C_2c_2k_1\|\varphi\|_\infty e^{2k_1\|\varphi\|_\infty}.
\]
In the latter estimates we have also used \eqref{f bound}.
It follows  that the conditions of Lemma \ref{A A' lemma} are satisfied with $\ve_1=R_1|z|$. Now we will estimate $|A'-B'|$. 
First, write 
\begin{eqnarray*}
|A'-B'|\leq\frac1{\mu_{\sig^k\om}(P)}\int_P\big|\cL_\om^{z,k}f-\cL_\om^{0,k}f\big|d\tilde m_{\sig^k\om}\\=
\frac1{\mu_{\sig^k\om}(P)}\int_P\big|\cL_\om^{0,k}\big(f(e^{zS_k^\om\varphi}-1)\big)|d\tilde m_{\sig^k\om}\\
\leq\|f\|_\infty \|e^{zS_k^\om\varphi}-1\|_\infty\frac1{\mu_{\sig^k\om}(P)}\int_P\cL_\om^{0,k}\textbf{1}d\tilde m_{\sig^k\om}
\leq M_1\|f\|_\infty \|e^{zS_k^\om\varphi}-1\|_\infty\frac{\tilde m_{\sig^k\om}(P)}{\mu_{\sig^k\om}(P)}\\\leq
 M_1Dc_2\int fd\tilde m_\om\,\cdot 2k_1e^{2k_1\|\varphi\|_\infty}\|\varphi\|_\infty|z|
=R_2|z|B
\end{eqnarray*}
where $D$ is defined by (\ref{D def}), $M_1$ is an upper bound on $\|\cL_\om^{0,k}\textbf{1}\|_\infty$ for $k_1\leq k\leq 2k_1$ (in fact, we can use Proposition \ref{Prop LY} to obtain an upper bound which does not depend on $k$ and $\om$)
 and
\[
R_2= M_1Da^{-1}2c_2k_1\|\varphi\|_\infty e^{2k_1\|\varphi\|_\infty}.
\]
We conclude now from Lemma \ref{A A' lemma} that 
\[
|\gam(\cL_\om^{z,k}f)-\gam(\cL_\om^{0,k}f)|\leq 2R_3(1-\del)^{-1}|z|\gam(\cL_\om^{0,k}f)
\]
where $R_3=\max(R_1,R_2)$.

Next, consider the case when $\gam$ have the form $\gam=\Gam_{x,\pm}$ for some $x\in Q_{\sig^k\om,j_{\sig^k\om}+1}$. Set
\begin{eqnarray*}
A=c\int \cL_\om^{z,k}fd\tilde m_{\sig^k\om},\,\, A'=\pm \cL_\om^{z,k}f(x),\\
B=c\int\cL_\om^{0,k}fd\tilde m_{\sig^k\om}\,\,\text{ and }\,\,B'=\pm\cL_\om^{0,k}f(x).
\end{eqnarray*}
Then $B>0$ and by (\ref{Smaller cone}) we have 
\[
|B'|\leq \del B.
\]
Similarly to the previous case, we have 
\[
|A-B|\leq R_4B|z|
\]
where $R_4=2c_2k_1\|\varphi\|_\infty$. Now we will estimate $|A'-B'|$. Using (\ref{f bound}) 
we have
\begin{eqnarray*}
|A'-B'|=|\cL_\om^{z,k}f(x)-\cL_\om^{0,k}f(x)|\leq \|f\|_\infty\|e^{zS_k^\om\varphi}-1\|_\infty\cL_\om^{0,k}\textbf{1}(x)\\
\\\leq c_2\int fd\tilde m_\om\,\cdot(2k_1|z|\|\varphi\|_\infty e^{2k_1\|\varphi\|_\infty} M_1)=BR_5|z| 
\end{eqnarray*}
where $R_5=2c_2k_1\|\varphi\|_\infty M_1 e^{2k_1\|\varphi\|_\infty}$ and
$M_1$ is an upper bound on $\|\cL_\om^{0,k}\textbf{1}\|_\infty$ for $k_1\leq k\leq 2k_1$.  
Since 
\[
|\gam(\cL_\om^{z,k}f)-\gam(\cL_\om^{0,k}f)|=|A-A'-(B-B')|,
\]
we conclude from Lemma \ref{A A' lemma} that 
\[
|\gam(\cL_\om^{z,k}f)-\gam(\cL_\om^{0,k}f)|\leq 2R_6(1-\del)^{-1}|z|\gamma (\cL_\om^{0,k})
\]
where $R_6=\max\{R_4,R_5\}$. 

Finally, we consider the case when $\gam=\Gam_{x,x'}$ for some distinct $x'$ and $x'$ which belong to the same floor of $\Del_{\sig^k\om}$.  
Set $d(x,x')=d_{\sig^k\om}(x,x')$,
\begin{eqnarray*}
A=b\int \cL_\om^{z,k}fd\tilde m_{\sig^k\om},\,\, A'=\frac{\cL_\om^{z,k}f(x)-\cL_\om^{z,k}f(x')}{d(x,x')},\\
B=b\int\cL_\om^{0,k}fd\tilde m_{\sig^k\om}\,\,\text{ and }\,\,B'=\frac{\cL_\om^{0,k}f(x)-\cL_\om^{0,k}f(x')}{d(x,x')}.
\end{eqnarray*}
Then, exactly as in the previous cases, $B>0$,  we have that $|B'|\leq \del B$,
\[
|\gamma (\cL_\om^{z,k}f)-\gamma(\cL_\om^{0,k}f)|=|A-A'-(B-B')|
\]
and 
\[
|A-B|\leq R_7B|z|
\]
where $R_7=2c_2b^{-1}+k_1R\|\varphi\|_\infty$.  Now we will estimate $|A'-B'|$.
Let $\ell$ be so that $x,x'\in\Del_{\sig^k\om,\ell}$ and write
$x=(x_0,\ell)$ and $x'=(x_0',\ell)$. Then $d_{\sig^k\om}(x,x')=\be^{\ell-m}d_{\sig^{m}\om}((x_0,m),(x_0',m))$ for any $0\leq m\leq\ell$.
If $k\leq \ell$ then for any complex $z$,
\[
\cL_\om^{z,k}f(x)=v_\ell^{-1}v_{\ell-k}e^{zS_k^\om\varphi(x_0,\ell-k)}f(x_0,\ell-k)
\]
and a similar equality hold true with $x'$ in place of $x$.
Set 
\begin{eqnarray*}
U(z)=f(x_0,\ell-k)e^{zS_k^\om\varphi(x_0,\ell-k)} \,\text{ and }\,  V(z)=f(x_0',\ell-k)e^{zS_k^\om\varphi(x_0',\ell-k)}
\end{eqnarray*}
and $W(z)=U(z)-V(z)$.
Then for any $z\in\bbC$ so that $|z|\leq1$ we have
\[
d(x,x')|A'-B'|=v_\ell^{-1}v_{\ell-k}|W(z)-W(0)|\leq |z|\sup_{|\zeta|\leq1}|W'(\zeta)|.
\]
Since the functions $u_\om$ and $f$ are locally Lipschitz continuous (uniformly in $\om$) we obtain that for any $\zeta$ so that $|\zeta|\leq1$,
\[
|W'(\zeta)|\leq C_1d(x,x')\|f\| \leq d(x,x')C_1(b+c_2)\int fd\tilde m_\om=d(x,x')C_1b^{-1}(b+c_2)B
\]
where $C_1$ depends only on $k_1$ and $\|\varphi\|_\infty$, and $d(x,x')=d_{\sig^k\om}(x,x')$.

Next, suppose that $k>\ell$, where $\ell$ is such that $x,x'\in\Del_{\sig^k\om,\ell}$.
The approximation of  $|A'-B'|$ in this case is carried out essentially as in the classical case of uniformly distance expanding maps, as described in the following arguments.
First, since $k>\ell$  we can write 
\[
F_\om^{-k}\{x\}=\{y\},\,\,F_\om^{-k}\{x'\}=\{y'\}
\] 
where both sets are at most countable, the map $y\to y'$ is bijective and satisfies that for all $0\leq q\leq k$,
\[
d_{\sig^q\om}(F_\om^qy,F_\om^qy')\leq \beta^{k-q} d(x,x')\leq d(x,x').
\]
Note also that the paring is done so that $(y,y')$ also belong to the same partition element in $\Del_\om$.
Then for any complex $z$ we have
\[
\cL_\om^{z,k}f(x)=v_\ell^{-1}\sum_{y}v(y)JF_\om^k(y)^{-1}e^{zS_k^\om\varphi(y)}f(y)
\] 
and 
\[
\cL_\om^{z,k}f(x')=v_\ell^{-1}\sum_{y'}v(y)JF_\om^k(y')^{-1}e^{zS_k^\om\varphi(y')}f(y')
\] 
where we note that $v(y)=v(y')$ since $y$ and $y'$ belong to the same floor. For any $y$ set 
\[
U_{y}(z)=JF_\om^k(y)^{-1}e^{zS_k^\om\varphi(y)}f(y)
\]
and 
\[
W_{y,y'}(z)=U_{y}(z)-U_{y'}(z).
\]
Then for any complex $z$ so that $|z|\leq1$ we have
\[
|W_{y,y'}(z)-W_{y,y'}(0)|\leq|z|\sup_{|\zeta|\leq 1}|W'_{y,y'}(\zeta)|.
\]
Since $JF_\om ^k$ satisfies (\ref{Distortion}) and $\varphi_\om$ and $f$ are locally Lipschitz continuous (uniformly in $\om$) we 
derive that 
\begin{equation}\label{Der Bound}
\sup_{|\zeta|\leq 1}|W'_{y,y'}(\zeta)|\leq C_2\|f\|d(x,x')(JF_\om^k(y)^{-1}+JF_\om^k(y')^{-1})
\end{equation}
for some constant $C_2$ which depends only on $\text{ess-sup}\|\varphi_\om\|, k_1$ and on $Q$ from (\ref{Distortion}).
Using that 
\[
\|f\|\leq(c_2+b)\int fd\tilde m_\om
\]
for some $c_2>0$
we derive now from (\ref{Der Bound})  that 
\begin{eqnarray*}
d(x,x')|A'-B'|=v_\ell^{-1}\left|\sum_{y}v(y)\big(W_{y,y'}(z)-W_{y,y'}(0)\big)\right|
\\\leq \big(|z|d(x,x')C_2\|f\|\big)v_\ell^{-1}\sum_{y}v(y)(JF_\om^k(y)^{-1}+JF_\om^k(y')^{-1})\\=
\big(|z|d(x,x')C_2\|f\|\big)\cdot\big(\cL_\om^{0,k}\textbf{1}(x)+\cL_\om^{0,k}\textbf{1}(x')\big)\leq E_1|z| B
\end{eqnarray*}
where $E_1=2M_1C_2b^{-1}(c_2+b)$ and $M_1$ is an upper bound of $\sup_n\|\cL_\om^{0,n}\textbf{1}\|_\infty$.
We conclude that there exists a constant $C_0$ so that for any $s\in\Gamma_\om$, $f\in\cC'$, $z\in\bbC$ and $k_1\leq k\leq 2k_1$,
\[
|\gamma(\cL_\om^{z,k}f)-\gam(\cL_\om^{0,k}f)|\leq C_0|z|\gamma(\cL_\om^{0,k}f).
\]
Let $r>0$ be any positive number so that 
\[
\del_r:=2C_0r\Big(1+\cosh\big(\frac12 d_0\big)\Big)<1.
\]
Then, by (\ref{Comp1}) and what proceeds it, (\ref{I}) and (\ref{II}) hold true $\bbP$-a.e.
for any $z\in\bbC$ with $|z|<r$ and $k_1\leq k\leq 2k_1$, and the proof of Theorem \ref{ConeThm} is complete.
\end{proof}


\section{Proofs of the limit theorems}\label{SecLimThms}
In this section we will work under Assumptions \ref{Ass Ap}, \ref{Ass Exp} and \ref{AsCov}. In particular  Theorem \ref{RPF} holds true.
Let $\varphi_\om:\Del_\om\to\bbR$, $\om\in\Om$ be a family of functions so that $\text{ess-sup }\|\varphi_\om\|_{Li}<\infty$ and $\varphi(\om,x)$ is measurable in both $\om$ and $x$. For $\bbP$-a.e. $\om$ we consider the functions
\[
S_n^\om \varphi=\sum_{j=0}^{n-1}\varphi_{\sig^j\om}\circ F_\om^j.
\]

\subsection{A Berry-Esseen theorem}
The proof of the first part  proceeds exactly as the proof of \cite[Theorem 2.5]{HafSDS}, and the proof of the second part is similar. For readers' convenience we will give the details of the second part, where is is enough to prove it in the case when $\mu_\om(S_n^\om \varphi)=0$ for any $n$ (i.e. when $\mu_\om(\varphi_\om)=0$). First, by (\ref{ExpCor}) applying \cite[Proposition 3.2]{EagHaf}  with $p_2=p_3=2$, $p_1=\infty$ and $M_j=(j+1)^{-2}$ and \cite[Proposition 3.3]{EagHaf} we indeed get (\ref{VarDif}). 

Next, using the properties of $\la_\om(z)$ one can define a branch $\Pi_{\om}(z)$ of $\ln\la_{\om}(z)$ in some deterministic neighborhood $U$ of $0$ so that $\Pi_\om(0)=0$ and $|\Pi_{\om}(z)|\leq c_0$ for some $c_0>0$. Set $\Pi_{\om,n}(z)=\sum_{j=0}^{n-1}\Pi_{\sig^j\om}(z)$. We claim first that 
\begin{equation}\label{C1}
\Pi_{\om,n}'(0)=0\,\,\text { and }\text{ess-sup }\sup_{n}|\Pi_{\om,n}''(0)-\Sig_{\om,n}^2|<\infty.
\end{equation}
In order to prove the first equality we first differentiate both sides of the identities $\nu_\om^{(z)}(h_\om^{(z)})=1$ and $\nu_\om^{(z)}(h_\om^{(0)})=\textbf{1}$ with respect to $z$ and then substitute $z=0$. This yields that 
\[
\nu_\om^{(0)}\left(\frac{d}{dz}h_\om^{(z)}\Big|_{z=0}\right)=0
\]
Next, we differentiate  the identity 
$\cL_\om^{z,n}(h_\om^{(z)})=\la_{w,n}(z)h_{\sig^n\om}^{(z)}$ with respect to $z$,  plug in $z=0$ and then integrate both resulting sides with respect to $\nu_\om^{(0)}=\tilde m_\om$. This yields that 
\[
\la_{w,n}'(0)=\tilde m_\om(h_\om^{(0)}S_n^\om \varphi)=\int S_n^\om \varphi d\mu_\om
\]
where we have used that $\mu_\om=h_\om d m_\om=\tilde h_\om d\tilde h_\om$ and that $h_\om^{(0)}=\tilde h_\om=h_\om/v$.
Since $\la_{\om,n}'(0)=\Pi_{\om,n}'(0)$ the proof of the claim is complete.
Now we will prove the inequality in (\ref{C1}). First, by iterating \eqref{L Du} and using that $\tilde h_\om=h_\om/v$, $\tilde m_\om=vdm_\om$ and $\mu_\om=h_\om dm_\om$, for any complex $z$ we have
\begin{equation}\label{CharFunc}
\mu_\om(e^{zS_n^\om \varphi})=\tilde m_\om\big(\cL_{\om}^{z,n}(\tilde h_\om)\big)=\tilde m_\om\big(\cL_{\om}^{z,n}(h_\om/v)\big).
\end{equation}
Using (\ref{RPF}) we can write
\begin{equation}\label{CharFunc1}
\tilde m_\om\big(\cL_{\om}^{z,n}(h_\om/v)\big)=\la_{\om,n}(z)\left(\tilde m_\om(h_{\sig^n \om}^{(z)})\nu_\om^{(z)}(\tilde h_\om)+\del_{\om,n}(z)\right)
\end{equation}
where $\del_{\om,n}(z)$ is an analytic function so that $|\del_{\om,n}(z)|\leq c\del^n$.
Let us now consider the analytic function $G_{\om,n}(z)=\tilde m_\om(h_{\sig^n \om}^{(z)})\nu_\om^{(z)}(\tilde h_\om)+\del_{\om,n}(z)$. Since  $\tilde h_\om=h_\om^{(0)}$ and $\tilde m_\om=\nu_\om^{(0)}$, using also \eqref{CharFunc} we conclude that $G_{\om,n}(0)=1$. Moreover, $G_{\om,n}$ is  bounded around the origin, uniformly in $\om$ and $n$, since $z\to h_\om^{(z)}$ and $z\to \nu_\om^{(z)}$ are uniformly bounded around the origin. Thus we can develop analytic branches of $\log G_{\om,n}(z)$ around the origin which vanish at $z=0$ and are uniformly bounded. Taking now the logarithms of both sides of  (\ref{CharFunc1}) and then considering the second derivatives at $z=0$, using the Cauchy integral formula we get that 
\begin{equation}\label{I.1}
\left|\text{Var}_{\mu_\om}(S_n^\om \varphi)-\Pi_{w,n}''(0)\right|\leq R
\end{equation}
where $R>0$ is some constant which does not depend on $n$, where we have used \eqref{CharFunc} to differentiate the left hand side.

Next, set $a_\om=m_\om(\Del_\om)$. Then there is a constant $C>1$ so that $1\leq a_\om\leq C$ for $\bbP$ a.e. $\om$. Now, for for any $z\in\bbC$,
\begin{equation}\label{Fib CLT rel 1}
\bar m_{\om}(e^{zS_n^\om \varphi})=a_\om^{-1}m_{\sig^n\om}(P_\om^{0,n}e^{zS_n^\om \varphi})=
a_\om^{-1}m_{\sig^n\om}(P_\om^{z,n}\textbf{1})=a_\om^{-1}\tilde m_{\sig^n\om}(\cL_\om^{z,n}(1/v)).
\end{equation}
Set $U=B(0,r)$, where $r$ comes from Theorem \ref{RPF}.
Let the analytic function $\varphi_{\om,n}:\to\bbC$ given by 
\begin{equation}\label{var phi def 1}
\varphi_{\om,n}(z)=\frac{\tilde m_{\sig^n\om}(\cL_\om^{z,n}(1/v))}{a_\om \la_{\om,n}(z)}.
\end{equation}
Then by (\ref{Fib CLT rel 1}) for any $z\in U$ and $n\geq1$,
\begin{equation}\label{Fib CLT rel 2}
\bar m_{\om}(e^{zS_n^\om \varphi})=
e^{\Pi_{\om,n}(z)}\varphi_{\om,n}(z).
\end{equation}
Next, by (\ref{I.1}) we have $\Pi_{\om,n}'(0)=0$
and therefore by (\ref{var phi def 1}), 
\begin{equation}\label{var phi ' 0 =0}
\varphi_{\om,n}'(0)=0.
\end{equation}

Now, we claim that there exists constants $A$ such that $\bbP$-a.s.
for all $n\in\bbN$ and  $z\in\bbC$ so that $|z|<r$ (i.e. $z\in U$) we have
\begin{equation}\label{var phi bound}
|\varphi_{\om,n}(z)|\leq A.
\end{equation}
Indeed, by (\ref{Exponential convergence}), there exist 
constants $A_1,k_1>0$ and $c\in(0,1)$ 
such that  for any $z\in U$
and $n\geq k_1$,
\begin{equation}\label{L z approx B.E.}
\left\|\frac{\cL_\om^{z,n}(1/v)}{\la_{\om,n}(z)}- h_{\sig^n\om}^{(z)}\nu_\om^{(z)}(1/v)
\right\|\leq A_1\del^{n}.
\end{equation}
The estimate (\ref{var phi bound}) follows now since $m_\om(\Del_\om)\leq C$, $\|\nu_\om^{(z)}\|\leq C$ and $\|h_\om^{(z)}\|\leq C$ for some $C>1$ and all $z$ in a neighborhood of $0$.

Next, by considering the Taylor expansion of $\varphi_{\om,n}$ of order $2$
we deduce from  (\ref{var phi ' 0 =0}) and (\ref{var phi bound})
that there exists a constant $B_1>0$ such that 
\begin{equation}\label{The const B}
|\varphi_{\om,n}(z)-\varphi_{\om,n}(0)|=|\varphi_{0,n}(z)-1|\leq B_1|z|^2
\end{equation}
for any $z\in\bbC$ so that $|z|\leq r/2$. 
Moreover, using (\ref{C1}) and (\ref{VarDif}) we see that there exist constants
$t_0,c_0>0$ such that $\bbP$-a.s. for any $s\in[-t_0,t_0]$ and a sufficiently large 
$n$,
\begin{equation}\label{Re press0}
\Big|\Pi_{\om,n}(is)+
\frac{s^2}2 v_{\om,n}\Big|\leq c_0|s|^3n+\frac12R_1s^2
\end{equation}
where $R_1$ is some constant and
we have also used that that $|\Pi_\om(z)|\leq c_0$ for some $c_0$ which does not depend on $\om$ and $z$. 
 Then, since  $v_{\om,n}$  grows linearly fast in $n$, we obtain from (\ref{Re press0}) that there exist  constants $t_0>0$ and $q>0$ so that for any $s\in[-t_0\sqrt{n},t_0\sqrt{n}]$ and 
 all sufficiently large $n$ we have
\begin{equation}\label{Re press1}
\Re\Big(\Pi_{\om,n}(is)\Big)\leq -qs^2\sqrt n.
\end{equation}
Next, by the Berry-Esseen inequality 
for any two distribution functions $F_1:\bbR\to[0,1]$
and $F_2:\bbR\to[0,1]$
with characteristic functions $\psi_1,\psi_2$, respectively, and $T>0$,
\begin{equation}\label{Essen ineq}
\sup_{x\in\bbR}|F_1(x)-F_2(x)|\leq \frac{2}{\pi}\int_{0}^T
\big|\frac{\psi_1(t)-\psi_2(t)}{t}\big|dt+\frac{24}{\pi T}\sup_{x\in\bbR}|F_2'(x)|
\end{equation}
assuming that $F_2$ is a  function with a bounded first
derivative. Let $\del_0>0$ and set $T_n=\del_0/\sqrt n$. For any real $t$ set 
$t_n=t/\sqrt{v_{\om,n}}$. Let $t\in[-T_n,T_n]$.  Then if $\del_0$ is small enough we have
by (\ref{Fib CLT rel 2}),
\begin{eqnarray}\label{We have}
|\bar m_\om(e^{it_nS_n^\om \varphi})-e^{-\frac12t^2}|\leq e^{\Re(\Pi_{\om,n}(it_n)}|\varphi_{\om,n}(it_n)-1|
\\+|e^{\Re(\Pi_{\om,n}(it_n))}-e^{-\frac12t^2}|:=I_1(n,t)+I_2(n,t).\nonumber
\end{eqnarray}
By (\ref{Re press1}) and (\ref{The const B}) we have
\[
I_1(n,t)\leq B_1e^{-qt^2}t^2/v_{\om,n}\leq C_\om e^{-qt^2}t^2n^{-1}.
\]
Using the mean value theorem, together with
(\ref{Re press0}) applied with $s=t_n$, taking into account (\ref{Re press1}) we derive that 
\[
I_2(n,t)\leq c_1v_{\om,n}^{-1}(|t|^3+t^2)e^{-c_2t^2}
\]
for some constants $c_1,c_2>0$. 
Let $F_1$ be the distribution function of $S_n^\om \varphi$ (w.r.t $\bar m_\om$), and let $F_2$ be the standard normal distribution.
Applying (\ref{Essen ineq}) with these functions and the above $T=T_n$ we obtain the second statement
with  $S_n^\om \varphi/\sqrt{v_{\om,n}}$ with respect to $\bar m_\om$. By using   \cite[Proposion 3.2]{EagHaf} we have that \[
\text{ess-sup }\sup_n|\bar m_\om(S_n^\om \varphi)-\mu_\om(S_n^\om \varphi)|=
\text{ess-sup }\sup_n|\bar m_\om(S_n^\om \varphi)|<\infty.
\]
Therefore, the difference between the centered and non-centered sum is $O(1/\sqrt n)$. 
Applying  \cite[Lemma 3.3]{HK-BE} with $a=\infty$ we complete the proof of the second part.
\qed

\subsection{The local CLT}
Since the CLT holds true, in both lattice and aperiodic cases, applying \cite[Theorem 2.2.3]{HK}, the local CLT's follows from (\ref{A3}), (\ref{A3'}), or their $\bar m_\om$-versions together with the estimates 
\[
|e^{\Pi_{\om,n}(it)}|=e^{\Re\left(\Pi_{\om,n}(it)\right)}\leq c_1e^{-c_2 nt^2} 
\]
which holds true for any $t\in[-\del,\del]$, a sufficiently small $\del>0$ and a sufficiently large $n$, where $c_1,c_2$ are positive constants. Indeed, in all four  local CLT's in question the characteristic function of the underlying sum is bounded from above around the origin by a constant times the function $|e^{\Pi_{\om,n}(it)}|$ (see (\ref{Fib CLT rel 2}) and its $\mu_\om$-version).
\qed

\subsubsection{On the verification of conditions (\ref{A3}) and (\ref{A3'})}\label{SecVer}
For uniformly random expanding maps (see  \cite[Ch. 5\& 7]{HK}) and for random uniformly hyperbolic maps \cite{DFGTV3},  conditions (\ref{A3}) and (\ref{A3'}) were verified  under certain assumption involving regularity properties of the random maps $f_\om$ and functions $u_\om$ around a periodic orbit of $\sig$, and other regularity assumptions on the behavior of the systems $(\Om,\cF,\bbP,\sig)$ aroud that periodic orbit (see \cite[Assumption 2.10.1]{HK}, \cite[Assumption 7.1.2]{HK} and \cite[Assumption 5.5]{HafEdge}). In this section we will extend this idea to random Young towers.

We assume here that $M_\om$ does not depend on $\om$ and that $(\Om,\cF,\bbP,\sig)$ is a product shift space, where $\Om=\Om_0^\bbZ$ is a topological space, $\cF$ contains all the Borel sets and $\bbP=P_0^\bbZ$ is a product measure.  Since in the applications in Section \ref{SecApp} we can only consider the case of i.i.d. maps, we will focus this case, even though it is possible to formulate results in more general circumstances. 
In this case we take $f_\om=f_{\om_0}$, where $\om=(\om_j)_{j\in\bbZ}$. We will also assume that $R_\om$ is a stopping time: for all $n$, $x$ so that $R_\om(x)=n$, we have $R_{\om'}(x)=n$  for evry $\om'\in\Om$ such that $\om'_j=\om_j$ for all $0\leq j<n$.
The following Assumption is our version of   \cite[Assumption 7.1.2]{HK} (or \cite[Assumption 5.5]{HafEdge} which is a more general version of it).

\begin{assumption}\label{AssFix}
(i) There is a point $\om_0\in\Om_0$ so that $P_0$ assigns positive mass to open neighborhoods of $\om_0$. 

(ii) The map $\om\to u_\om$ is continuous at the point  $a:=(...,\om_0,\om_0,\om_0,...)=\om_0^\bbZ$. Moreover, 
for any $n$, the operator $\cP_{\om,n}$ given by
\[
\cP_{\om,n}g(x_0)=\sum_{y: f_\om^n y=x_0, R_{\om}(y)=n}g(y)/Jf^n(y)=\cP_{\om}^{0}(\bbI(R_\om=n)g)(x_0)
\] 
is continuous in $\om$ at the point $a$.

(iii) The spectral radius of the deterministic transfer operator $\cR_{it}:=\cL_{a}^{it}$ is strictly less than $1$ for any $t\not=0$ in the aperiodic case, or for any nonzero $t\in[-\pi/h,\pi/h]$ in the lattice case (equivalently, the spectral radius of $\cP_a^{it}$ with respect to the norm $\|g\|=\|g\|_s+\|g\|_h$ defined in Section \ref{Pre} is less than $1$ for non-zero $t$'s in the above domains).
\end{assumption}
We note that because of the product structure we build our condition around a fix point of $\sig$, and not around a general periodic point (as in \cite{HK}), but, of course, considering periodic points is also possible. In this case we should just replace $\cL_a^{it}$ with $\cL_{a}^{it, n_0}$, where $n_0$ is the period of $a$, and all the continuity and regularity properties should hold true for points belonging to the finite periodic orbit of $a$. 

The second condition holds true when $f_{\om_0}=f_{\om_0'}$ if $\om_0'$ is close enough to $\om_0$. This happens when $\Om_0$ is a countable alphabet and $P_0(\{\om_0\})>0$. More general type of continuity of $f_{\om'}$ in $\om'$ around $\om_0$ can be considered.
The third condition is just a standard apriodicity (or maximality) assumption on the deterministic Young tower $(\Del_a, F_a)$.

\begin{proposition}
Suppose that  Assumption \ref{AssFix} holds true. Then
for $\bbP$-a.a. $\om$ the  left hand sides of (\ref{A3}) and (\ref{A3'}) decay exponentially fast to $0$, with either $\mu_\om$ or $\tilde m_\om$ in place of $\mu_\om$ (and for any appropriate set $J$).
\end{proposition}

\begin{proof}
First, using the uniform exponential tails and (\ref{Distortion0}), we have that for any $M$ and $t\in\bbR$, uniformly in $\om$,
\begin{equation}\label{Approx M}
\left\|\cL_{\om}^{it}-\cL_{\om}^{it, \leq M}\right\|\leq (1+|t|)c_1e^{-c_2 M}
\end{equation}
where $c_1, c_2>0$ are constants and $\cL_{\om}^{it, \leq M}(g)=\cL_{\om}^{it}(g\bbI(R_{\om}\leq M))$. 

Next, let $J$ be a compact subset of either $\bbR\setminus\{0\}$ (in the aperiodic case) or $[-\pi/h,\pi/h]\setminus\{0\}$ (in the lattice case). Let $B_J\geq1$ be so that 
\[
\sup_{n\geq 1}\sup_{t\in_J}\|\cL_{\om}^{it,n}\|\leq B_J.
\]
As noted before, such a constant exists in view of the Lasota-Yorke inequality.
Let $s$ be so large so that 
\[
\sup_{t\in J}\|\cR_{it}^s\|\leq \frac1{4 B_J}.
\]
Such an $s$ exists in view of Assumption \ref{AssFix} (iii).
Let $\ve>0$. Then by (\ref{Approx M}) and the compactness of $J$ there exists $M=M_\ve$ so that 
 for any $\om$ we have
\[
\sup_{t\in J}\|\cL_{\om}^{it}-\cL_{\om}^{it,\leq M}\|<\ve.
\]
Therefore, there is a constant $A_{j,s}>0$ so that 
\[
\sup_{t\in J}\|\cL_{\om}^{it,s}-\cL_{\om}^{it,\leq M,s}\|<A_{J,s}\ve.
\]
where 
\[
\cL_{\om}^{it,\leq M,s}=\prod_{j=0}^{s-1}\cL_{\sig^j\om}^{it,\leq M}.
\]
Next, by  Assumption \ref{AssFix} (ii) there is a neighborhood $U$ of $a$ so that for any $\om\in U$ we have
\[
\sup_{t\in J}\|\cL_{\om}^{it,\leq M}-\cL_{a}^{it,\leq M}\|<\ve.
\]
Set $V=\bigcap_{j=0}^{s-1}\sig^{-j}U$. Then $V$ is an open neighborhood of $a$, and so $\bbP(V)>0$ (since $P_0$ assigns positive mass to open sets containing $\om_0$). 
It follows that  there is a constant $C_{J,s}>0$ so that for any $\om\in V$ we have
\[
\sup_{t\in J}\left\|\cL_{a}^{it,\leq M,s}-\cL_{\om}^{it,\leq M, s}\right\|\leq C_{J,s}\ve.
\]
By taking a sufficiently small $\ve$ we get that
\[
\sup_{\om\in V}\sup_{t\in J}\left\|\cL_{\om}^{it,s}-\cR_{it}^s\right\|<\frac1{2B_J}.
\]

Finally, by Birkhoff's ergodic theorem and the Kac formula, for $\bbP$-a.a. $\om$ there is an infinite sequence $n_1<n_2<...$ so that 
\[
\lim_{m\to\infty}n_m/m=1/\bbP(V)>0.
\]
Therefore, there is a constant $c>0$ so that, $\bbP$-a.s. when $n$ is large enough we can partition $\cL_{\om}^{it,n}$ into at least $cn$ blocks so that the norm of the odd blocks does not exceed $B_J$, while the norm of the even blocks does not exceed $\frac12 B_J$ (we can take $c=P(V)/2s$). Therefore, $\bbP$-a.s. for any $n$ large enough we have
\[
\sup_{t\in J}\|\cL_{\om}^{n,it}\|\leq D_J2^{-cn}
\]
and the proof of the proposition is complete.
\end{proof}

\begin{remark}
When (\ref{A3}) and (\ref{A3'}) hold true then we can also get first order Edgeworth expansions in a similar way to  \cite{DH1} and \cite{HafEdge}.
\end{remark}
\subsection{Large and moderate deviations principles: proofs}
Relying on the G\"artner-Ellis Theorem and on (\ref{Exponential convergence}), (\ref{I.1}) and that 
\[
\left|\mu_\om(S_n^\om \varphi)-\bar m_\om(S_n^\om \varphi)\right|\leq C,
\]
the proof of Theorems \ref{MD} and \ref{LD} proceed  exactly as in \cite{HafSDS} (in our case the variance grows linearly fast).
The main idea in the proof is that, using (\ref{Exponential convergence}) when $z\in\{\zeta\in\bbC: |\zeta|\leq \del\}$ (where $\del$ is small enough) we get that  for both choices $\ka_\om=\mu_\om$ and $\ka_\om=\bar m_\om$ we have
\[
\ln \ka_\om(e^{z(S_n^\om \varphi-\mu_\om(S_n^\om \varphi))})=
\sum_{k=0}^{n-1}\la_{\sig^k\om}(z)+O(1).
\]
Diving by $n$ and taking the limit as $n\to\infty$ yields Theorem \ref{LD}. In Theorem \ref{MD} we have a speed function which is of sublinear order in $n$. In this case. using second order Taylor expansions of the function $z\to\la_\om(z)$ (using \eqref{C1}) and then applying the G\"artner-Ellis Theorem
yields Theorem \ref{MD} exactly as in  \cite[Theorem 2.8]{HafSDS}.

\subsection{additional limit theorems}
We can also obtain the local CLT and the large and moderate deviations principles for vector valued random observables $\varphi_\om$. The proofs are very close to the corresponding proofs in \cite{DH1}, and so they are not provided. Moreover, using the ideas in \cite{Annealed}, under appropriate conditions we can also get a local CLT, a Berry-Esseen theorem and a Renewal theorem for the sums $S_n\varphi =\sum_{j=0}^{n-1}\varphi\circ T^j$, where $\varphi(\om,x)=\varphi_\om(x)$, $T(\om,x)=(\sig\om, F_\om x)$ is the skew product and $(\om,x)$ is distributed according to $\mu=\int \mu_\om dP(\om)$. In the applications in Section \ref{SecApp}, all of the above results translate into corresponding results with $f_\om$ instead of $F_\om$ and with the equivariant measures $\mu_\om$ discussed there.




\begin{acknowledgment}
I would like to thank D.  Dragi\v cevi\' c  for reading carefully a preliminary version of the paper and for some useful comments. 
\end{acknowledgment}


\begin{thebibliography}{Bow75}

\bibliographystyle{alpha}
\itemsep=\smallskipamount

\bibitem{ABR}
J. Alv\'es, W.Bahsoun and M. Ruziboev, {\em Almost sure rates of mixing for partially hyperbolic attractors}, preprint, arXiv 1904.12844.


\bibitem{ADLP}
J. Alv\'es, C.Dias, S.Luzzatto and V. Pinheiro
{\em SRB measures for partially hyperbolic systems whose central direction is weakly expanding},
J. Eur. Math. Soc. 19, 2911-2946 (2017) 


\bibitem{Aimino}
R. Aimino, M. Nicol and S. Vaienti. {\em Annealed and quenched limit theorems for random expanding dynamical systems}, Probab. Th. Rel. Fields 162, 233-274, (2015).

\bibitem{Arno}
P. Arnoux and A.Fisher, {\em Anosov families, renormalization and non-stationary
subshifts}, Erg. Th. Dyn. Syst., 25 (2005), 661-709



\bibitem{BBM}
V. Baladi, M. Benedicks, and V. Maume-Deschamps, {\em Almost sure rates of mixing for i.i.d. unimodal
maps}. Ann. Sci. \'Ecole Norm. Sup. 35, 77--126 (2002).


\bibitem{BB}
W. Bahsoun and C. Bose, {\em Mixing rates and limit theorems for random intermittent maps}. Nonlinearity,
29(4):1417–1433, 2016.

\bibitem{BBD}
W. Bahsoun, C. Bose, and Y. Duan., {\em Decay of correlation for random intermittent maps}. Nonlinearity,
27(7):1543–1554, 2014.

\bibitem{BBR}
W. Bahsoun, C. Bose, and M. Ruziboev, {\em Quenched decay of correlations for slowly mixing systems}.
Trans. Amer. Math. Soc., 2019.


\bibitem{CP} Z.Coelho and  W.Parry, {\em Central limit asymptotics for shifts of finite type}, Israel J. Math. 69, (1990), no. 2, 235

\bibitem{Conze}
J-P Conze and A. Raugi, {\em Limit theorems for sequential expanding dynamical
systems}, AMS 2007.


\bibitem{Cr}
H. Crauel and F. Flandoli, {\em Attractors for random dynamical systems}, Probab. Theory Related Fields,
100, 365-393 (1994).

\bibitem{DZ} M. Demers and H. Zhang, \emph{A functional analytic approach to perturbations of the Lorentz Gas}, Comm. Math. Phys. \textbf{324} (2013), 767--830.

\bibitem{DFGTV1} D. Dragi\v cevi\' c, G. Froyland, C. Gonzalez-Tokman and S. Vaienti, \emph{Almost Sure Invariance Principle for random piecewise expanding maps}, Nonlinearity \textbf{31} (2018), 2252--2280.

\bibitem{DFGTV2} D. Dragi\v cevi\' c, G. Froyland, C. Gonzalez-Tokman and S. Vaienti, \emph{A spectral approach for quenched limit theorems for random expanding dynamical systems}, Comm. Math. Phys. \textbf{360} (2018), 1121--1187.

\bibitem{DFGTV3} D. Dragi\v cevi\' c, G. Froyland, C. Gonzalez-Tokman and S. Vaienti, \emph{A spectral approach for quenched limit theorems for random hyperbolic dynamical systems}, Trans. Amer. Math. Soc.\textbf{373} (2020), 629--664.



\bibitem{DH}
D. Dragi\v cevi\' c, Y. Hafouta.
\newblock{A vector-valued almost sure invariance principle for random hyperbolic and piecewise-expanding maps},
\newblock preprint. https://arxiv.org/abs/1912.12332.


\bibitem{DH1}
D.  Dragi\v cevi\' c  and Y. Hafouta, {\em Limit theorems for random expanding or hyperbolic dynamical systems and
vector-valued observables}, Ann. Henri Poincaré 21 (2020), 3869--391.




\bibitem{Dub1}
L. Dubois, {\em Projective metrics and contraction principles for
complex cones}, J. London Math.
Soc. 79 (2009), 719-737.

\bibitem{Dub2}
L. Dubois, {\em An explicit Berry-Esseen bound for uniformly expanding maps on the interval}, Israel J. Math. 186 (2011), 221-250.

\bibitem{Du} Z. Du. {\em On mixing rates for random perturbations}. PhD thesis, National University of Singapore, 2015.


\bibitem{GH}
Y. Guivar\'ch and J. Hardy, {\em Th\'eor\`emes limites pour une classe de cha\^ines de Markov et
applications aux diff\'eomorphismes d'Anosov},
 Ann. Inst. H. Poincar\'e Probab. Statist. 24 (1988), no. 1, 73-98.
 


\bibitem{GO}
S. Gou\"ezel, {\em Berry-Esseen theorem and local limit theorem for non uniformly expanding maps}, Annales de l'Institut Henri Poincar\"e, Prob. and Stat. 997–1024 (2005).


\bibitem{HK-BE}
 Y. Hafouta and Yu. Kifer, {\em Berry-Esseen type estimates
for nonconventional sums}, Stoch. Proc. Appl. 126 (2016), 2430-2464.


\bibitem{HK}
Y. Hafouta and Yu. Kifer, {\em Nonconventional limit theorems and random dynamics}, 
World Scientific, Singapore, 2018.


\bibitem{Annealed} Y. Hafouta, {\em Limit theorems for some skew products with mixing base maps}, 
 Ergodic Theory Dynam, Volume 41 , Issue 1 , January 2021 , pp. 241 - 271


\bibitem{HafEdge}
Y. Hafouta, {\em Asymptotic moments and Edgeworth expansions for some processes in random dynamical environment}, J. Stat. Phys 179, 945-971 (2020).


\bibitem{HafSDS}
Y. Hafouta, {\em Limit theorems for some time dependent expanding dynamical systems}, Nonlinearity 33, 6421.



\bibitem{LLT2020}
Y. Hafouta, {\em  A local limit theorem for number of multiple recurrences generated by non-uniformly hyperbolic or expanding maps}, accepted for publication in Journal d'Analyse Mathematique, available on
 arXiv:2003.08528v1.

\bibitem{EagHaf}
Y. Hafouta, {\em On Eagleson's theorem in the non-stationary setup}, accepted for publication in Bull. Lon. Math. Soc.	available on arXiv:2004.09333.

\bibitem{HeSt}
O.Hella and M. Stenlund, {\em Quenched normal approximation for random sequences of transformations},  J Stat Phys (2019) doi:10.1007/s10955-019-02390-5.  

\bibitem{HydPsyl} 
N. Haydn and Y. Psiloyenis, {\em Return times distribution for Markov towers with decay of correlations}, Nonlinearity, Volume 27, Number 6 (2014).

\bibitem{Hyd}
Nicolai Haydn, Matthew Nicol, Andrew Török and Sandro Vaienti, {\em Almost sure invariance principle for sequential and non-stationary dynamical systems}, Trans. Amer. Math. Soc. 369 (2017), 5293-5316.





 
\bibitem{HH}
H. Hennion and L. Herv\'e, {\em Limit Theorems for Markov Chains and Stochastic Properties of Dynamical
Systems by Quasi-Compactness}, Lecture Notes in Mathematics vol. 1766, Springer, Berlin, 2001.



\bibitem{K96}
Y. Kifer,
\newblock Perron-Frobenius theorem, large deviations, and random perturbations in random environments,
\newblock {\em Math. Z.}, 222 (1996), 677--698.

\bibitem{K98} Y. Kifer, \emph{Limit theorems for random transformations and processes in random environments}, Trans. Amer. Math. Soc. \textbf{350} (1998), 1481--1518.


\bibitem{Zemer}
A Korepanov, Z Kosloff, I Melbourne
{\em Martingale-coboundary decomposition for families of dynamical systems}
 Annales de l’Institut Henri Poincar\'e C, Analyse non lin\'eaire 35, no. 4 (2018) 859-885.











 











\bibitem{M-D}
V\'eronique Maume-Deschamps, {\em Projective metrics and mixing properties on towers}, Trans. Amer. Math. Soc. 353, 3371-3389 (2001).


\bibitem{MN}
I. Melbourne, M. Nicol, {\em Large deviations for nonuniformly hyperbolic systems}, Trans. Amer. Math. Soc. 360, 6661-6676 (2008).


\bibitem{NSV}
P.~N{\'a}ndori, D.~Sz{\'a}sz, and T.~Varj{\'u}.
\newblock A central limit theorem for time-dependent dynamical systems,
\newblock {\em J. Stat. Phys.} 146 (2012), 1213--1220.

\bibitem{Nicol}
M. Nicol, A. Torok, S. Vaienti, {\em Central limit theorems for sequential and random intermittent dynamical systems}, 
Ergodic Theory Dynam. Systems, 38, pp. 1127-1153, 2016.

\bibitem{Rug}
H.H. Rugh, {\em Cones and gauges in complex spaces: Spectral gaps and complex
 Perron-Frobenius theory}, 
Ann. Math. 171 (2010), 1707-1752.


\bibitem{Su}
Y. Su
\newblock Random Young Towers and Quenched Limit Laws,
\newblock preprint, https://arxiv.org/pdf/1907.12199.pdf

\bibitem{Y1}
L.S. Young, {\em Statistical properties of dynamical systems with some hyperbolicity}, 
 Ann. Math. 7 (1998) 585-650.

\bibitem{Y2}
L.S. Young, {\em  Recurrence time and rate of mixing}, Israel J. Math. 110 (1999) 153-88.










\end{thebibliography}
\end{document}